\documentclass[11pt,reqno]{amsart}
\usepackage{fullpage}
\usepackage{xcolor}
\usepackage[T1]{fontenc}                          
\usepackage{amsfonts}
\usepackage[utf8]{inputenc}                       
\usepackage{comment}                              
\usepackage{mparhack}                             
\usepackage{amsmath,amssymb,amsthm,mathrsfs,eucal}
\usepackage{amsfonts}
\usepackage{enumerate}
\usepackage{enumitem}
\usepackage{booktabs}                             
\usepackage{graphicx,subfig}   
\usepackage{cancel}
\usepackage{wrapfig}                            
\usepackage[bookmarks=true, colorlinks=true]{hyperref}
\hypersetup{urlcolor=blue,citecolor=red,linkcolor=black}
\usepackage{cleveref}
\usepackage{soul}

\usepackage{appendix}

\usepackage{bm}                                   
\usepackage{dsfont}

\newtheorem{defn}{Definition}[section]
\newtheorem{thm}{Theorem}[section]
\newtheorem{prop}{Proposition}[section]
\newtheorem{lem}{Lemma}[section]
\newtheorem{cor}{Corollary}[section]

\newtheorem{rem}{Remark}[section]

\numberwithin{equation}{section}

\DeclareMathOperator*{\argmin}{argmin}

\newcommand{\mP}{{\mathcal{P}}}

\newcommand{\mh}{\mathcal{H}}
\newcommand{\mf}{\mathcal{F}}

\newcommand{\mprd}{{\mathcal{P}(\Rd)}}

\newcommand{\mptrd}{{\mathcal{P}_2(\Rd)}}

\newcommand{\mpdtard}{{\mathcal{P}_{2}^a(\Rd)}}
\newcommand{\mw}{\mathcal{W}}

\newcommand{\rhote}{\rho_{\tau}^{\varepsilon}}

\newcommand{\rhotne}{\rho_{\tau,\varepsilon}^{n}}

\newcommand{\rhotnne}{\rho_{\tau,\varepsilon}^{n+1}}

\newcommand{\R}{\mathbb{R}}
\newcommand{\Rd}{{\mathbb{R}^{d}}}
\newcommand{\Rn}{{\mathbb{R}^{n}}}
\newcommand{\Rdd}{{\mathbb{R}^{2d}}}

\newcommand{\rhoe}{\rho^\varepsilon}

\newcommand{\supp}{\mathrm{supp}}

\makeatletter
\@namedef{subjclassname@2020}{\textup{2020} Mathematics Subject Classification}
\makeatother

\def\XXint#1#2#3{{\setbox0=\hbox{$#1{#2#3}{\int}$}
  \vcenter{\hbox{$#2#3$}}\kern-.5\wd0}}

\title{Nonlocal approximation of nonlinear diffusion equations
}
\date{}

\begin{document}

\author{José Antonio Carrillo \and Antonio Esposito \and Jeremy Sheung-Him Wu}
\address{J. A. Carrillo, A. Esposito -- Mathematical Institute, University of Oxford, Woodstock Road, Oxford, OX2 6GG, United Kingdom.}

\email{carrillo@maths.ox.ac.uk}
\email{antonio.esposito@maths.ox.ac.uk}

\address{J. S.-H. Wu -- Mathematical Sciences Building, University of California, Los Angeles, CA 90095, United States.}

\email{jeremywu@math.ucla.edu}

\begin{abstract}
We show that degenerate nonlinear diffusion equations can be asymptotically obtained as a limit from 
a class of nonlocal partial differential equations. The nonlocal equations are obtained as gradient flows of interaction-like energies approximating the internal energy. We construct weak solutions as the limit of a (sub)sequence of weak measure solutions by using the Jordan-Kinderlehrer-Otto scheme from the context of $2$-Wasserstein gradient flows. Our strategy allows to cover the porous medium equation, for the general slow diffusion case, extending previous results in the literature. As a byproduct of our analysis, we provide a qualitative particle approximation.
\end{abstract}





\maketitle

\section{Introduction}

Nonlinear diffusion equations are ubiquitous in several real world applications. They were introduced to analyse gas expansion in a porous medium, groundwater infiltration, and heat conduction in plasmas, to name a few applications in physics. These applications drove the first rigorous mathematical results by Zel'dovich and Kompaneets in \cite{zel1950towards} and Barenblatt in \cite{barenblatt1953one} regarding important particular weak solutions of nonlinear diffusion equations with homogeneous nonlinearity. The general filtration equation was then first developed in \cite{Ka74}. The use of these equations in oil recovery software is extensive nowadays. Another source of applications of this family of equations arises from population models in mathematical biology: ecological models \cite{boi,TBL06,BCM07} derived from probabilistic interpretations \cite{Oelschlger1990LargeSO,FigPhi2008}, volume effect in Keller-Segel type models \cite{painter2002volume,gamba2003,calvez2006volume}, volume exclusion in cell-cell adhesion models \cite{CarrilloMurakawaCellAdhesion,CCY19}, and many others.

Although a rigorous mathematical theory has been extensively provided over the years \cite{Vaz07, otto,CJMTU}, there are particular aspects of renewed interest in view of novel applications as well as advances in mathematics. For instance, their derivation from interacting particles, with a distinction between deterministic and stochastic methods, has recently attracted attention for its implications in derivation of models in mathematical biology \cite{CarrilloMurakawaCellAdhesion} and data science \cite{blob_weighted_craig}. We take advantage of the gradient flow structure of nonlinear diffusions \cite{otto} to connect with nonlocal interaction equations. In fact, we rigorously derive particle approximations of nonlinear diffusions from these variational considerations by approximating their energy functional completing the approach started in \cite{Patacchini_blob19}. 

For ease of presentation, let us focus on more standard diffusion equations. Let $m\ge1$ and consider the equation
\[
\partial_t\rho = \Delta \rho^m,
\]
which is better known as the heat equation for $m=1$, or the porous medium equation (PME) in the case $m>1$. A comprehensive study of the above PDE can be found, e.g., in the book of V\'azquez, \cite{Vaz07}. Owing to the advances in optimal transport theory, \cite{V1,V2,S}, starting from the seminal works of Jordan, Kinderlehrer, and Otto, \cite{JKO98,otto}, such diffusion equations are known to be $2$-Wasserstein gradient flows for a specific choice of the energy functional. More precisely, the previous equation can be written as  
\begin{equation}
\label{eq:gf-general}
\left\{
\begin{array}{c}
    \partial_t\rho+\nabla\cdot\left(\rho v\right)=0,\\[2mm]
    v=-\nabla\frac{\delta \mh_m}{\delta\rho},
\end{array}
\right.
\end{equation}
being $\frac{\delta \mh_m}{\delta\rho}$ the first variation of the energy functional
\begin{equation}\label{eq:diffusion_energy}
    \mh_m[\rho]=
    \begin{cases}
    \int_\Rd\rho(x)\log\rho(x)dx \qquad &m=1\\[2mm]
    \frac{1}{m-1}\int_\Rd\rho^m(x)dx &m>1
    \end{cases}.
\end{equation}
In \cite{JKO98} the equation of interest was the linear Fokker--Planck equation, while Otto focused on the porous medium equation in \cite{otto}. Afterwards, a $2$-Wasserstein gradient flow approach has been extended to other PDEs, in particular those modelling nonlocal interaction, \cite{CMcCV03,CMcCV06,AGS,CDFFLS}. The latter equation is of the form \eqref{eq:gf-general} with $v=-\nabla\frac{\delta \mw}{\delta\rho}$ and
\begin{equation}\label{eq:interaction_energy}
    \mw[\rho]=\frac{1}{2}\int_\Rd W*\rho(x)d\rho(x).
\end{equation}

Recent works in the literature show a rigorous and fascinating connection between the two energies above for $m>1$ in \eqref{eq:diffusion_energy} and the corresponding dynamics, by means of gradient flow techniques, c.f.~\cite{Patacchini_blob19,BE22}. More precisely, exploiting the so-called \textit{blob method} developed in \cite{Craig_blob2016}, one can notice already at a formal level that an appropriate regularisation of $\mh_m$ transforms a diffusion equation (which is local) into an interaction PDE (which is nonlocal) by choosing a delocalising kernel. For simplicity, let $m=2$ and consider a standard family of non-negative radial mollifiers $V_\varepsilon(x) = \varepsilon^{-d}V_1(x/\varepsilon)$ for $\varepsilon>0$ on $\Rd$.  Using the commutativity of convolution with even functions such as $V_\varepsilon$, it is indeed not difficult to see
\[
\mh_2[V_\varepsilon*\rho]=\int_\Rd (V_\varepsilon*\rho)^2(x)dx=\int_\Rd(V_\varepsilon*V_\varepsilon)*\rho(x)d\rho(x)=\int_\Rd W_\varepsilon*\rho(x)d\rho(x)=2\mw_\varepsilon[\rho],
\]
by setting $W_\varepsilon:=V_\varepsilon*V_\varepsilon$. This observation sheds light on the aforementioned link between local and nonlocal PDEs. As a natural byproduct such a connection provides a rigorous particle approximation for a class of nonlinear diffusion equations. More precisely, this hinges on deterministic approaches for nonlocal interaction equations, since, \textit{particles are solutions}, i.e. the following empirical measure $\rho_t^N$ is a weak solution of~\eqref{eq:gf-general} with $v=-\nabla \frac{\delta \mw}{\delta \rho}$ and $\mw$ as in~\eqref{eq:interaction_energy}
\[
\rho_t^N=\frac{1}{N}\sum_{i=1}^N\delta_{X_i(t)},
\]
where, for any $i=1,\dots,N$, $X_i(t)$ solves the ODE
\[
\dot{X}_i(t)=-\frac{1}{N}\sum_{j}\nabla W(X_i(t)-X_j(t)).
\]
Further details on this aspect can be found, e.g., in \cite{CDFFLS,CCH14}, and in \cite{DFF,DFESPSCH21} in case of systems of nonlocal PDEs. This structure is advantageous for the computational approximation of continuous solutions to~\eqref{eq:gf-general}. The main issue when diffusion is present is that particles do not remain particles. Indeed, if the initial datum is a Dirac delta, we have an immediate smoothing effect, excluding measure solutions. However, numerical evidence of these deterministic particle methods \cite{Patacchini_blob19} show that this can be achieved, as we shall see later on.

In this manuscript, we consider a general class of internal energy functionals $\mf : \mptrd \to (-\infty,+\infty]$ given by
\[
\mf[\rho] := \left\{
\begin{array}{cl}
\int_{\Rd}F(\rho(x))\, dx, 	&\rho\ll \mathrm{Leb}(\Rd)	\\[2mm]
+\infty, 	&\text{otherwise}
\end{array}
\right.,
\]
where we identify the measure $\rho$ with its density $\rho(x)$ if it is absolutely continuous with respect to Lebesgue measure and $\mptrd$ denotes the set of probability measures with finite second order moment. We define the regularised internal energy functional $\mf^\varepsilon : \mP_2(\Rd)\to (-\infty,+\infty]$ given by
\[
\mf^\varepsilon [\rho] := \int_{\Rd}F(V_\varepsilon*\rho(x)) \, dx,
\]
which gives rise to a class of nonlocal PDEs
\begin{equation}\label{eq:nlie-class}
    \partial_t\rho=\nabla\cdot(\rho\nabla V_\varepsilon*F'(V_\varepsilon*\rho)). \tag{NLE}
\end{equation}
The functional $\mf$ includes $\mh_m$, but it is not limited to it, c.f.~\Cref{sec:preliminary}. The reader is invited to verify
\[
\frac{\delta \mf^\varepsilon}{\delta \rho}(\rho) = V_\varepsilon * \left[\frac{\delta\mf}{\delta \rho}(V_\varepsilon * \rho)\right],
\]
which motivates the consideration of~\eqref{eq:nlie-class} as the 2-Wasserstein gradient flow of $\mf^\varepsilon$. Following the strategy proposed in~\cite{BE22}, defining the pressure by $P(x):=x F'(x)-F(x)$, as in~\cite{McC97,CMcCV03,AGS}, we construct weak solutions of the nonlinear diffusion equation
\begin{equation}\label{eq:nonlinear-diffusion}
    \partial_t\rho=\Delta P(\rho)\tag{DE}
\end{equation}
as a limit of a sequence of weak measure solutions of \eqref{eq:nlie-class}, in case $F$ behaves like power laws of porous medium type, for $m>1$.

The blob method for diffusion was first introduced in \cite{Patacchini_blob19} for diffusion equations with the addition of local and nonlocal drifts. Let us mention that a similar approach was used on the previous work~\cite{Craig_blob2016} approximating nonlocal equations with singular kernels by smooth kernels. The authors in \cite{Patacchini_blob19} consider a slightly different regularisation of the internal energy which is better for numerical purposes, see \cite[Eq. (6)]{Patacchini_blob19}. Despite this difference, the gradient flow perspective remains at the forefront of their and our present work. The corresponding nonlocal gradient flow is indeed different from \eqref{eq:nlie-class}, c.f.~\cite[Eq. (8)]{Patacchini_blob19}, but it coincides with ours in case $m=2$ for the energy $\mh_2$. In \cite{Patacchini_blob19}, $\Gamma$-convergence of the regularised energy, as well as that of minimisers is proven for $m\ge1$. The authors show that stability of gradient flows in the $\varepsilon\to0$ can be established for $m\ge2$ using the framework introduced by Sandier and Serfaty in \cite{Sandier_Serfaty2004,Serfaty_gammaconv_2011} and the concept of $\lambda$-gradient flows developed in \cite{AGS}. This strategy requires to verify additional assumptions which are only known to hold in the case $m=2$ for an initial datum with finite second order moment and log-entropy, i.e. $\mh_1[\rho_0]<\infty$. The result for $m=2$ was previously proven in \cite{Lions_MasGallic_2001}, however on a bounded domain with periodic boundary conditions. The blob method in \cite{Patacchini_blob19} is a deterministic particle method for linear and nonlinear diffusion on $\Rd$. Numerical simulations in \cite[Section 6]{Patacchini_blob19} suggest that the particle approximation remains valid even when $\mh_1[\rho_0]=\infty$. Relaxing the condition $\mh_1[\rho_0]<\infty
$ and rigorously proving a quantitative particle approximation is still an open problem, and left for future research.

In the case $m=2$, in the same spirit of \cite{Patacchini_blob19,blob_weighted_craig}, the authors in \cite{BE22} construct weak solutions of the quadratic porous medium equation as a localising limit ($\varepsilon\to0$) of a sequence of weak measure solutions of the nonlocal interaction equation \eqref{eq:nlie-class}, for $F(x)=x^2$. The authors work directly at the level of the (nonlocal) equations by means of a time-discretisation scheme which allows to work with lack of convexity, as for instance in the case of cross-diffusion systems, or even PDEs with no purely gradient flow structure. As in~\cite{Patacchini_blob19}, finite initial log-entropy is required, thus excluding particle approximation. However, simultaneously to~\cite{BE22}, the authors in~\cite{blob_weighted_craig} focus on a weighted (quadratic) porous medium equation which is relevant, e.g. in sampling --- the weight, $\bar\rho$ in their notations, represents a target probability measure to be approximated from specific samples drawn from it. The blob method is indeed useful to develop a deterministic particle approximation for the weighted porous medium equation, and, as a byproduct, it provides a way to quantize a target $\bar\rho$ in the long-time behaviour. We stress that also in this work it is essential to assume $\mh_1[\rho_0]<\infty$, however using again $\lambda$-convexity of the regularised energy one can achieve a rigorous particle approximation as consequence of $\lambda$-stability (or contractivity) of Wasserstein gradient flows, as in \cite{AGS}. This means one can achieve, so far, a qualitative result, as the initial datum needs to be approximated fast enough, c.f.~\cite[Theorem 1.4]{blob_weighted_craig}. To the best of our knowledge, a quantitative result has not been achieved yet in more than one dimension. Still in one space dimension, the authors in \cite{Daneri_Radici_Runa_JDE} introduce a deterministic particle approximation for aggregation-diffusion equations, including the porous medium equation for the subquadratic ($1<m<2$) and superquadratic ($m>2$) cases. This approach, however, is limited to one space dimension. All the previous three works do not make use of gradient flow techniques. Indeed, other attempts for a particle method have been proposed in the literature. Let us mention two simultaneous numerical methods for linear diffusion $(m=1)$ \cite{Degond_Mustieles90,Russo90}. In one dimension and for nonlinear diffusions, there are other numerical methods based on the PDE satisfied by the transporting maps, see \cite{GosseToscani06,CM09,CRW16}. A nice survey of most of the available numerical methods for these families of equations can be found in \cite{CMWreview}.

\smallskip
Further related to particle methods, we mention the seminal paper by Oelschl\"{a}ger, \cite{Ol}, where a stochastic particle approximation is proven for classical and positive solutions of the quadratic porous medium equation in $\Rd$, and for weak solutions in one dimension, and the recent results in \cite{CDHJ21} for systems. In \cite{FigPhi2008} very weak solutions of the viscous porous medium equation ($m>1$) are studied as a limit of a sequence of distributions of the solutions to nonlinear stochastic differential equations generalising previous results \cite{oelschlaeger2001,morale2005interacting,FigPhi2008}. In \cite{Philipowski2007} strong $L^1$-solutions, c.f.~\cite{Vazquez92_PM_introduction}, of the quadratic porous medium equation are derived from a stochastic mean field interacting particle system with the addition of a vanishing Brownian motion.

\smallskip
Our strategy is different from the aforementioned stochastic approaches as it is based on an optimal transport approach avoiding the addition of higher regularity induced by the (vanishing) viscosity method. We consider a time-discretisation of \eqref{eq:nlie-class} à la Jordan-Kinderlehrer-Otto (JKO), c.f.~\cite{JKO98}. This method provides uniform bounds on the approximating sequence in terms of the associated energy and second order moments. Although the sequence solving \eqref{eq:nlie-class} is only a measure, we are able to prove strong $L^m$-compactness of a smoother sequence of solutions for the $\varepsilon\to0$ limit by using the so-called \textit{flow interchange} technique, c.f.~\cite{MMCS}. More precisely, one of our main contributions is to construct weak solutions of 
    \begin{equation}\label{eq:pme}
    \partial_t\rho=\Delta(\rho^m). \tag{PME}
\end{equation}
as a subsequential $\varepsilon\to0$ limit of weak measure solutions to
\begin{equation}\label{eq:nlie}
    \partial_t\rhoe=\frac{m}{m-1}\nabla\cdot(\rhoe\nabla V_\varepsilon*(V_\varepsilon*\rhoe)^{m-1}), \tag{NLE-m}
\end{equation}
for all $m>1$. The same result is proven also for \eqref{eq:nlie-class} and \eqref{eq:nonlinear-diffusion}. In particular, this extends~\cite{BE22} to the case $m>2$, which is not trivial in view of the nonlinearities involved, and to a class of general nonlinear diffusion function. In~\cite{Patacchini_blob19}, their gradient flow convergence result for $m>2$ was conditional on a uniform $BV$ bound for $\rho^\varepsilon$ while we make no such assumptions here. Furthermore, we are also able to treat the case $1<m<2$ which is more challenging due to the lack of regularity at zero.

As a byproduct of our analysis we obtain an existence result for nonlocal diffusion equations related to a nonlocal internal energy functional. In particular, we are able to construct weak solutions to~\eqref{eq:nlie} via the JKO scheme for $m>1$. While this may not be surprising, this is the first result in this direction to the best of our knowledge. We also provide a particle approximation for \eqref{eq:nlie-class} in case $F$ behaves like power laws, for $m>1$. This result is purely qualitative, and quantitative estimates are not proven. Finally, we stress that the strategy we use to construct weak solutions does not require convexity of the internal energy, thus allowing to extend this method to non-convex energies, e.g. nonlinear cross-diffusion systems, see \cite{BE22}. We leave the extension to systems for a future work as it deserves a deeper analysis.

The case $m=1$, i.e. linear diffusion, is not completely covered in our theory, due to the lack of control on the compactness near the logarithmic singularity in the gradient flow approach. More precisely, our strategy does provide an approximating scheme, validated numerically in~\cite{Patacchini_blob19}, but we are not able at this stage to identify the limit as solution of the heat equation. Indeed, the logarithmic singularity cannot be coped with for the case $m=1$ when the mollifier $V_1$ is compactly supported. This is indeed one of the reasons we did not assume $V_1$ is compactly supported in the case $m\ge2$. Similar difficulties are found for the Landau equation \cite{Landau} in plasma physics, for which efficient deterministic particle methods preserving all the properties of the Landau equation at the discrete level were introduced in \cite{CHWW20} using the same strategy as in this work. Moreover such an approximated Landau equation has been analytically studied in \cite{CDDW20,CDW22} showing the existence of solutions for the approximated problems where $V_1(x)=e^{-|x|}$ with an appropriate mollification at the origin. The particular non-compactly supported kernel is crucial in the detailed estimates performed in \cite{CDDW20}. Dealing with the logarithmic singularity in these problems is a challenging open problem.

\subsection*{Structure of the manuscript} 
\Cref{sec:preliminary} sets the assumptions, notations, and definitions we use in this paper. At the end of~\Cref{sec:preliminary}, we state the precise results obtained once the appropriate notions of solutions are introduced. \Cref{sec:nlie} focuses on the construction of weak solutions $\rho^\varepsilon$ to~\eqref{eq:nlie-class} (c.f.~\Cref{thm:exist_nlie-class}) based on the JKO scheme~\cite{JKO98}. \Cref{sec:compact} discusses the strong compactness criteria used to construct a limit $\rho$ (which is the candidate weak solution to~\eqref{eq:nonlinear-diffusion}) from the sequence $\rho^\varepsilon$. \Cref{sec:nonlocal-to-local-limit} verifies that the limit $\rho$ is a weak solution to~\eqref{eq:nonlinear-diffusion} (c.f.~\Cref{thm:exist_nonlinear_diffusion}) by passing to the limit $\varepsilon\to 0$ from~\eqref{eq:nlie-class}. In~\Cref{sec:convexity}, we sketch the ideas behind the proofs of~\Cref{thm:uniqueness}, which gives conditions for uniqueness of solutions to~\eqref{eq:nlie-class}, and~\Cref{cor:particle_approx}, which provides a particle approximation to~\eqref{eq:nonlinear-diffusion}. Finally, \Cref{sec:proofs} collects various technical results which, possibly with minor adaptations, already exist in the literature.

\section{Preliminaries and results}\label{sec:preliminary}

The mollifying sequence is generated by $V_\varepsilon(x) = \varepsilon^{-d}V_1(x/\varepsilon)$ for $\varepsilon>0$. We assume that the generating function $V_1$ satisfies
\begin{enumerate}[label=\textbf{(V)}]
	\item $V_1\in C_b(\Rd;[0,+\infty))\cap C^1(\Rd)$, $\|V_1\|_{L^1}=1$, $V_1(x)=V_1(-x)$, $\int_\Rd|x|^2V_1(x)\,dx<+\infty$, $|\nabla V_1|\in L^1(\Rd)$, and $|\nabla V_1(x)|\le C(1+|x|)$.  	\label{ass:v1}
\end{enumerate}

Depending on the results we prove, we assume the function $F:[0,+\infty) \to (-\infty,+\infty]$ satisfies some combination of the following assumptions:
\begin{enumerate}[label=\textbf{(F\arabic*)}]
	\item \label{ass:AGS-F} $F$ is a proper, convex, and lower semicontinuous function such that
	\[
	F(0) = 0,\quad \liminf_{s\uparrow+\infty}\frac{F(s)}{s}= +\infty, \quad \liminf_{s\downarrow 0}\frac{F(s)}{s^\alpha}>-\infty, \quad \text{for some }\alpha > \frac{d}{d+2}.
	\]
	\item \label{ass:diff-F} $F \in  C^1([0,+\infty))$.
	\item \label{ass:diff-reg} $F \in C([0,+\infty))\cap C^2((0,+\infty))$.
\end{enumerate}
\begin{enumerate}[label=\textbf{($\textbf{F}_m$)}]
	\item \label{ass:F-m} There exist $c_1, \, c_2>0$ and $m\ge1$ such that $c_1 x^{m-2} \le F''(x) \le c_2 x^{m-2}$ for all $x> 0$.
\end{enumerate}
\begin{rem}[Comments on the assumptions]
\label{rem:F-assumptions}
\ref{ass:AGS-F} is lifted directly from~\cite[Example 9.3.6]{AGS} so that $\mf$ enjoys certain properties; it is well-defined and the associated JKO scheme is well-posed c.f.~\cite{JKO98}. 

For the reader's convenience, we observe the condition $\liminf_{s\downarrow 0}\frac{F(s)}{s^\alpha}>-\infty$ for some $\alpha>\frac{d}{d+2}$ ensures (c.f.~\cite[Remark 9.3.7]{AGS} and~\Cref{lem:lower-bound-F}) that $F^-(\rho) \in L^1(\Rd)$ whenever $\rho \in \mP_2(\Rd)$ is absolutely continuous with respect to Lebesgue measure. In particular, on any sublevel subset $\{\rho \in \mptrd \, | \, m_2(\rho) \le C\}$, the functional $\mf$ is uniformly bounded below.

The superlinear growth (c.f.~\cite[Remark 9.3.8]{AGS}) and convexity ensure that $\mf$ is lower semicontinuous in $\mP_1(\Rd)$.

Assumption \ref{ass:diff-reg} mainly refers to energies lacking regularity at the origin as in the case $F(x)=x\log x$. We stress that \ref{ass:diff-F} is used to construct solutions to~\eqref{eq:nlie-class} in~\Cref{sec:nonlocal-to-local-limit}, however it is not used to derive the compactness estimates in~\Cref{sec:compact}. Conversely, \ref{ass:diff-reg} is used for the compactness estimates in~\Cref{sec:compact} but is not assumed to construct solutions to~\eqref{eq:nlie-class}. The motivating examples which satisfy all of~\ref{ass:AGS-F}, \ref{ass:diff-F}, \ref{ass:diff-reg}, and~\ref{ass:F-m} are power laws $F(x) = \frac{1}{m-1}x^m$ for $m>1$.

A further discussion can be found after the statements of~\Cref{thm:exist_nlie-class} and~\Cref{thm:exist_nonlinear_diffusion}.
\end{rem}

Throughout the manuscript we will denote by $\mP(\Rd)$ the set of probability measures on $\Rd$, for $d\in\mathbb{N}$, and by $\mP_p(\Rd):=\{\rho\in\mP(\Rd):m_p(\rho)<+\infty\}$, being
$m_p(\rho):=\int_\Rd|x|^p\,d\rho(x)$ the $p^{\mathrm{th}}$-order moment of $\rho$, for $1\le p<\infty$. We shall use $\mP_p^a(\Rd)$ for elements in $\mP_p(\Rd)$ which are absolutely continuous with respect to the Lebesgue measure. For $p=2$, the $2$-Wasserstein distance between $\mu_1,\mu_2\in \mptrd$ is
\begin{equation}\label{wass}
d_W^2(\mu_1,\mu_2):=\min_{\gamma\in\Gamma(\mu_1,\mu_2)}\left\{\int_{\Rdd}|x-y|^2\,d\gamma(x,y)\right\},
\end{equation}
where $\Gamma(\mu_1,\mu_2)$ is the class of all transport plans between $\mu_1$ and $\mu_2$, that is the class of measures $\gamma\in\mP(\Rdd)$ such that, denoting by $\pi_i$ the projection operator on the $i$-th component of the product space, the marginality condition
$$
(\pi_i)_{\#}\gamma=\mu_i \quad \mbox{for}\ i=1,2
$$
is satisfied. In the expression above, marginals are the push-forward of $\gamma$ through $\pi_i$. For a measure $\rho\in\mP(\Rd)$ and a Borel map $T:\Rd\to\Rn$, $n\in\mathbb{N}$, the push-forward of $\rho$ through $T$ is defined by
$$
 \int_{\Rn}f(y)\,dT_{\#}\rho(y)=\int_{\Rd}f(T(x))\,d\rho(x) \qquad \mbox{for all Borel functions $f$ on}\ \Rn.
$$
Setting $\Gamma_0(\mu_1,\mu_2)$ as the class of optimal plans, i.e. minimizers of \eqref{wass}, the $2$-Wasserstein distance can be written as
$$
d_W^2(\mu_1,\mu_2)=\int_{\Rdd}|x-y|^2\,d\gamma(x,y), \qquad \gamma\in\Gamma_0(\mu_1,\mu_2).
$$
We denote the $1$-Wasserstein distance with $d_1$ and it is defined by
\begin{align}\label{eq:1wass}
    d_1(\mu_1,\mu_2):=\min_{\gamma\in\Gamma(\mu_1,\mu_2)}\left\{\int_{\Rdd}|x-y|\,d\gamma(x,y)\right\}.
\end{align}
We refer the reader to \cite{AGS,V2,S} for further details on optimal transport theory and Wasserstein spaces. 

Below we recall the concepts of solutions used throughout the manuscript, distinguishing between measure and weak solutions. 

\begin{defn}[Weak measure solution to \eqref{eq:nlie-class}]\label{def:weak-meas-sol-F}
Suppose $F$ satisfies~\ref{ass:AGS-F} and~\ref{ass:diff-F}. An absolutely continuous curve $\rho^\varepsilon:[0,T]\to\mptrd$, mapping $t\in[0,T]\mapsto\rho_t^\varepsilon\in\mptrd$, is a weak measure solution to \eqref{eq:nlie-class} if, for every $\varphi\in C^1_c(\Rd)$ and any $t\in[0,T]$, it holds
\begin{equation}\label{eq:weak-form-F}
\int_\Rd\varphi(x)d\rho_t^\varepsilon(x)\!-\!\int_\Rd\varphi(x)d\rho_0(x)=-\int_0^t\int_\Rd\nabla\varphi(x)\cdot[\nabla V_\varepsilon*F'(V_\varepsilon*\rho_s)](x)d\rho_s^\varepsilon(x)ds.
\end{equation}
\end{defn}

\begin{defn}[Weak measure solution to \eqref{eq:nlie}]\label{def:weak-meas-sol}
An absolutely continuous curve $\rho^\varepsilon:[0,T]\to\mptrd$, mapping $t\in[0,T]\mapsto\rho_t^\varepsilon\in\mptrd$, is a weak measure solution to \eqref{eq:nlie} for $m>1$ if, for every $\varphi\in C^1_c(\Rd)$ and any $t\in[0,T]$, it holds
\begin{equation}\label{eq:weak-form}
\int_\Rd\varphi(x)d\rho_t^\varepsilon(x)\!-\!\int_\Rd\varphi(x)d\rho_0(x)=-\frac{m}{m-1}\int_0^t\int_\Rd\nabla\varphi(x)\cdot[\nabla V_\varepsilon*(V_\varepsilon*\rho_s)^{m-1}](x)d\rho_s^\varepsilon(x)ds.
\end{equation}
\end{defn}

\begin{rem}
By considering fixed $\varepsilon>0$ and the corresponding scaling for $V_1$ satisfying~\ref{ass:v1}, the driving velocity field satisfies
\begin{equation}
\begin{aligned}
\int_0^T\int_\Rd|[\nabla V_\varepsilon*(V_\varepsilon*\rho_t^\varepsilon)^{m-1}](x)|d\rho_t^\varepsilon(x)dt&\le\int_0^T\iint_\Rdd|\nabla V_\varepsilon(x-y)|(V_\varepsilon*\rho_t^\varepsilon)^{m-1}(y)\,dy\,d\rho_t^\varepsilon(x)\,dt\\
&= \int_0^T \int_{\Rd} (|\nabla V_\varepsilon| * \rho_t^\varepsilon)(y) (V_\varepsilon * \rho_t^\varepsilon)^{m-1}(y)dy \, dt\\
&\le \int_0^T \|V_\varepsilon * \rho_t^\varepsilon\|_{L^\infty}^{m-1}\left(\int_{\Rd} (|\nabla V_\varepsilon| * \rho_t^\varepsilon)(y) dy\right) \, dt \\
&\le \varepsilon^{-md}\|V_1\|_{L^\infty} \int_0^T \||\nabla V_\varepsilon| * \rho_t^\varepsilon\|_{L^1}dt \\
&\le \varepsilon^{-md}\|V_1\|_{L^\infty} T \|\nabla V_\varepsilon\|_{L^1}   \\
&= \frac{T}{\varepsilon^{md+1}}\|V_1\|_{L^\infty}\|\nabla V_1\|_{L^1} < \infty.
\end{aligned}
\end{equation}
\cite[Lemma 8.2.1]{AGS} provides the existence of a continuous representative for distributional solutions of continuity equations with velocity fields in $L^1([0,T];L^1(\rho_t))$. This justifies ~\Cref{def:weak-meas-sol} in the sense that the right-hand side of~\eqref{eq:weak-form} is well-defined. Note that a similar computation holds true for the velocity field in \eqref{eq:nlie-class} by applying~\Cref{lem:bdd-comp-F-V}, thus justifying~\Cref{def:weak-meas-sol-F} in the sense that the right-hand side of~\eqref{eq:weak-form-F} is well-defined.
\end{rem}

\begin{defn}[Weak solution to \eqref{eq:pme}]\label{def:sol-pme2}
A weak solution to the Cauchy problem for $m>1$
\begin{equation*}
    \begin{cases}
    \partial_t\rho=\Delta\rho^m\\
    \rho(0,\cdot)=\rho_0
    \end{cases}\tag{PME}
\end{equation*}
on the time interval $[0,T]$ with initial datum $\rho_0\in\mpdtard\cap L^m(\Rd)$ is an absolutely continuous curve $\rho\in C([0,T];\mP_2(\Rd))$ satisfying the following properties:
\begin{enumerate}
    \item for almost every $t\in[0,T]$ the measure $\rho(t)$ has a density with respect to the Lebesgue measure, still denoted by $\rho(t)$, such that $\rho\in L^\infty([0,T];L^m(\Rd))$ and $\nabla \rho^{\frac{m}{2}}\in L^2([0,T];L^2(\Rd))$;
    \item for any $\varphi\in C^1_c(\R^d)$ and all $t\in[0,T]$ it holds
 \begin{align*}
     \int_\Rd\varphi(x)\rho(t,x)\,dx= \int_\Rd\varphi(x)\rho_0(x)\,dx-\int_0^t\int_\Rd \nabla \varphi(x)\cdot \nabla \rho(s,x)^{m}\,dx\,ds;
 \end{align*}
\item $\rho^{-\frac{1}{2}}|\nabla \rho^m|\in L^1([0,T];L^2(\Rd))$.
\end{enumerate}
\end{defn}
\begin{rem}\label{rem:chain_rule_weak_form}
For the sake of clarity we point out the weak solution we obtain initially is
\[
\int_\Rd\varphi(x)\rho(t,x)\,dx=\int_\Rd\varphi(x)\rho_0(x)\,dx-2\int_0^t\int_\Rd\rho(s,x)^{\frac{m}{2}} \nabla \varphi(x)\cdot \nabla \rho(s,x)^{\frac{m}{2}}\,dx\,ds.
\]
The chain rule in Sobolev spaces gives sense to $\nabla\rho^m$ in $L^1(\Rd)$, hence the more standard concept of weak solution for porous medium equation. A further application of the chain rule identifies $\nabla \rho^{m}=\frac{m}{m-1}\rho \nabla \rho^{m-1}$, for $m\ge 2$; the same result, however, does not hold in the case $1<m<2$. Further details are provided in the proof of Theorem~\ref{thm:exist_nonlinear_diffusion} in Section~\ref{sec:nonlocal-to-local-limit}. Finally, the last condition in Definition~\ref{def:sol-pme2} is a consequence of uniqueness of very weak solutions, cf.~\cite{DalKen84}, and the theory in~\cite{AGS}.
\end{rem}
Equally, the same concept is extended to general diffusion equations.

\begin{defn}[Weak solution to \eqref{eq:nonlinear-diffusion}]\label{def:sol-nonlinear-diffusion}
Let $F$ satisfy \ref{ass:AGS-F}, \ref{ass:diff-F}, \ref{ass:diff-reg}, and \ref{ass:F-m} for some $m> 1$. A weak solution to the Cauchy problem
\begin{equation*}
    \begin{cases}
    \partial_t\rho=\Delta P(\rho)\\
    \rho(0,\cdot)=\rho_0
    \end{cases}\tag{DE}
\end{equation*}
on the time interval $[0,T]$ with initial datum $\rho_0\in\mpdtard$ such that $\mf[\rho_0]<\infty$ is an absolutely continuous curve $\rho\in C([0,T];\mP_2(\Rd))$ satisfying the following properties:
\begin{enumerate}
    \item for almost every $t\in[0,T]$ the measure $\rho(t)$ has a density with respect to the Lebesgue measure, still denoted by $\rho(t)$, such that $\rho\in L^\infty([0,T];L^m(\Rd))$ and $\nabla \rho^{\frac{m}{2}}\in L^2([0,T];L^2(\Rd))$;
    \item for any $\varphi\in C^1_c(\R^d)$ and all $t\in[0,T]$ it holds
\begin{align*}
        \int_\Rd\varphi(x)\rho(t,x)\,dx= \int_\Rd\varphi(x)\rho_0(x)\,dx-\int_0^t\int_\Rd \nabla \varphi(x)\cdot \nabla P(\rho(s,x))\,dx\,ds;
    \end{align*}
\item $\rho^{-\frac{1}{2}}|\nabla P(\rho)|\in L^1([0,T];L^2(\Rd))$.    
\end{enumerate}
\end{defn}
With the previous definitions, we are ready to state the results of this manuscript.
\begin{thm}[Existence for~\eqref{eq:nlie-class}]
\label{thm:exist_nlie-class}
Fix $\varepsilon>0$ and let $V_1$, the generator of the mollifying sequence $V_\varepsilon(x) = \varepsilon^{-d}V_1(x/\varepsilon)$, satisfy~\ref{ass:v1}. Let $F$ satisfy~\ref{ass:AGS-F} and~\ref{ass:diff-F} and suppose $\mf^\varepsilon[\rho_0]<+\infty$. Then, there exist weak measure solutions $\rho^\varepsilon$ to~\eqref{eq:nlie-class} such that $\rhoe(0) = \rho_0$.
\end{thm}
\begin{thm}[$\lim_{\varepsilon\to 0}\eqref{eq:nlie-class} = \eqref{eq:nonlinear-diffusion}$]
\label{thm:exist_nonlinear_diffusion}
Let $F$ satisfy~\ref{ass:AGS-F}, \ref{ass:diff-F}, \ref{ass:diff-reg}, and~\ref{ass:F-m} for some $m>1$. Suppose $\rho_0\in \mpdtard$ such that $\mf[\rho_0]<\infty$ and $V_1$ satisfies~\ref{ass:v1}. In the case $1<m<2$, assume further that $\supp V_1 \subset B_R$ for some $R>0$. Let $\rho^\varepsilon$ be a sequence of weak measure solutions to~\eqref{eq:nlie-class} from~\Cref{thm:exist_nlie-class} with initial condition $\rho^\varepsilon(0) = \rho_0$. Then, the sequence $\rho^\varepsilon$ converges narrowly to the unique weak solution $\rho$ of~\eqref{eq:nonlinear-diffusion} as $\varepsilon\downarrow 0$.
\end{thm}
\begin{rem}
    In the case $F$ is a power law given by $F(x) = \frac{1}{m-1}|x|^m$ for some $m>1$, all of~\ref{ass:AGS-F}, \ref{ass:diff-F}, \ref{ass:diff-reg}, and~\ref{ass:F-m} are fulfilled; Theorem~\ref{thm:exist_nonlinear_diffusion} holds for \eqref{eq:pme}.

    At first glance, our compactness estimates only show that a subsequence of $\rho^\varepsilon$ converges narrowly to $\rho$. However, we appeal to~\cite{Brezis_Crandall_79,DalKen84,Vaz07} which imply that weak solutions to~\eqref{eq:nonlinear-diffusion} and~\eqref{eq:pme} are unique. Hence, the entire sequence converges.
\end{rem}
In~\Cref{thm:exist_nlie-class}, the construction of weak measure solutions $\rho^\varepsilon$ to~\eqref{eq:nlie-class} leverages the JKO scheme~\cite{JKO98}. Just at the level of the JKO scheme, only~\ref{ass:AGS-F} is required (c.f.~\Cref{prop:en-ineq-mom-bound}) for which all of the regularised R\'enyi entropies $\mh_m^\varepsilon[\rho] = \mh_m[V_\varepsilon*\rho]$ for any $m\ge 1$ are admissible. In fact, assumption~\ref{ass:diff-F} enters only when verifying $\rho^\varepsilon$ is a weak measure solution of~\eqref{eq:nlie-class} (c.f.~\Cref{sec:nlie}). This excludes $F(x) = x\log x$, but all the power laws for $m>1$ are permitted in this consistency result. Moreover, the assumption that~\ref{ass:F-m} holds for some $m>1$ in~\Cref{thm:exist_nonlinear_diffusion} is only used to verify that the limit $\rho$ is a weak solution to~\eqref{eq:nonlinear-diffusion}. On the other hand, the construction of the limit $\rho$ from the sequence $\rho^\varepsilon$ allows to relax assumption~\ref{ass:F-m} to any $m\ge 1$ provided the initial condition $\rho^\varepsilon(0) = \rho_0$ belongs in $L^m\cap L\log L$ (c.f.~\Cref{sec:compact}), thus including all of the regularised R\'enyi entropies $\mh_m^\varepsilon$. To summarise in the specific case of $\mf^\varepsilon =\mh_m^\varepsilon$ as the regularised energy, the construction of curves $\rho^\varepsilon$ and $\rho$ without consideration of the respective equations~\eqref{eq:nlie} and~\eqref{eq:pme} can be done for any $m\ge 1$. However, our technique requires $m>1$ to verify that $\rho^\varepsilon$ is a weak measure solution of~\eqref{eq:nlie}. Moreover, when $1<m<2$, we insist that the generator, $V_1$ of the mollifying sequence, satisfies~\ref{ass:v1} and has compact support (in the case $m\ge 2$ only~\ref{ass:v1} is required). It is certainly interesting to investigate how we can close this gap to $m=1$ and we leave this direction for future research.

In~\Cref{thm:exist_nonlinear_diffusion} we prove that the solutions $\rho^\varepsilon$ to~\eqref{eq:nlie-class} coming from the construction in~\Cref{thm:exist_nlie-class} converge to $\rho$, the unique weak solution of~\eqref{eq:nonlinear-diffusion}. It is natural to ask whether other solutions $\tilde{\rho}^\varepsilon$ to~\eqref{eq:nlie-class} (not necessarily those constructed via the JKO scheme c.f.~\Cref{sec:nlie}) also converge to $\rho$. Actually, under additional assumptions on the nonlinearity $F$ and the mollifier $V$, the sequence $\rho^\varepsilon$ is unique.
\begin{thm}[Uniqueness of solutions to~\eqref{eq:nlie-class}]
\label{thm:uniqueness}
Let $F$ satisfy~\ref{ass:AGS-F}, \ref{ass:diff-F}, \ref{ass:diff-reg}, and~\ref{ass:F-m} for some $m>1$. Assume $V_1$ satisfies~\ref{ass:v1}, $V_1\in C^2(\Rd)$, and $D^2V_1\in L^\infty(\Rd)$. Then, the weak measure solution $\rho^\varepsilon$ in~\Cref{thm:exist_nlie-class} is unique among absolutely continuous curves $\rho:[0,T]\to \mP_2(\R^d)$ satisfying~\eqref{eq:nlie-class} in the sense of Definition~\ref{def:weak-meas-sol-F}.
\end{thm}
The following concluding result is completely analogous to Theorem 1.2 of~\cite{blob_weighted_craig}.
\begin{cor}[Particle approximation to~\eqref{eq:nonlinear-diffusion}]
\label{cor:particle_approx}
Let $F$ satisfy~\ref{ass:AGS-F}, \ref{ass:diff-F}, \ref{ass:diff-reg}, and~\ref{ass:F-m} for some $m>1$. Assume $V_1$ satisfies~\ref{ass:v1}, $V_1\in C^2(\Rd)$, and $D^2V_1\in L^\infty(\Rd)$. In the case $1<m<2$, assume moreover that $\supp V_1 \subset B_R$ for some $R>0$. For any $t\in[0,T]$, $N\in\mathbb{N}$, the empirical measure $\rho^N_\varepsilon(t) = \frac{1}{N}\sum_{j=1}^N \delta_{x^j_\varepsilon(t)}$ is a weak solution to~\eqref{eq:nlie-class} provided the particles satisfy the following ODE system
\[
\dot{x}^i_\varepsilon(t) = - \nabla \int_{\R^d} V_\varepsilon(x^i_\varepsilon(t) - y) F'\left(
\frac{1}{N}\sum_{j=1}^NV_\varepsilon(y - x^j_\varepsilon(t))
\right)dy \quad \forall i=1,\dots, N.
\]
Suppose that (up to a subsequence) as $\varepsilon\to0$ there exist $N=N(\varepsilon)\to+\infty$ such that
\[
e^{-\lambda_F^\varepsilon t} d_W(\rho_\varepsilon^N(0),\rho(0))\to0, \qquad \mbox{for }\,
\lambda_F^\varepsilon\approx-\varepsilon^{-2-d(m-1)},\quad t\in[0,T],
\]
with $\rho_0\in \mpdtard$ such that $\mf[\rho_0]<\infty$ and $T>0$. Then $\rho_\varepsilon^N(t)$ converges narrowly to a weak solution of~\eqref{eq:nonlinear-diffusion}, $\rho(t)$, for any $t\in[0,T]$.
\end{cor}
In view of~\Cref{cor:particle_approx} and~\cite{blob_weighted_craig}, if the initial distribution of particles $x_\varepsilon^i$ is cleverly chosen (so that $d_W(\rho_\varepsilon^N(0),\rho(0)) = O(1/N)$), then one can take $N = o\left(e^{-1/\varepsilon^{2+d(m-1)}} \right)$ to fulfill the hypothesis on the initial condition. However, it was also suggested in~\cite{Craig_blob2016,blob_weighted_craig} by numerical evidence that a much smaller number of particles $N\sim \varepsilon^{-1.01}$ for $m=2$ in one dimension still yields good accuracy. Bridging this gap between theory and practice is left for future investigation.

\section{Results on the nonlocal equation}\label{sec:nlie}

In this section we focus on \eqref{eq:nlie-class}. We show existence of weak measure solutions by means of the JKO scheme~\cite{JKO98} which is needed to derive uniform bounds for the nonlocal-to-local limit proven in~\Cref{sec:nonlocal-to-local-limit}. Although this is not the main purpose of the paper, and it may be unsurprising, this is indeed an existence result for weak measure solutions to a class of nonlocal PDEs, including nonlocal interactions but not limited to this case. To the best of our knowledge this is the first general result in this context --- the structure of $\mf^\varepsilon$ does not fit in the classical framework of functionals considered in~\cite{AGS}. Note that we do not require the functional to satisfy convexity, for instance as in \cite{AGS,CDFFLS}.

We consider initial data $\rho_0 \in \mptrd$ such that $\sup_{\varepsilon>0}\mf^\varepsilon[\rho_0]<+\infty$. In the case of nonlinear diffusion equations, $F(x) = \frac{1}{m-1}|x|^m$ with $m>1$ we denote the corresponding energy functionals by
\[
\mh_m^\varepsilon[\rho] := \frac{1}{m-1}\int_{\Rd}|V_\varepsilon * \rho(x)|^m \, dx.
\]
\begin{rem}
\label{rem:uniform-bound-initial-energy}
In the case of power laws $F(x) = \frac{1}{m-1}|x|^m$ for $m>1$, the condition $\sup_{\varepsilon>0}\mh_m^\varepsilon[\rho_0]<+\infty$	is guaranteed when $\rho_0\in\mptrd\cap L^m(\Rd)$. More precisely, Young's convolution inequality gives
	\begin{align*}
		\mh_m^\varepsilon[\rho_0]&=\frac{1}{m-1}\int_{\Rd}|V_\varepsilon * \rho_0(x)|^m\,dx =\frac{1}{m-1}\|V_\varepsilon*\rho_0\|_{L^m(\Rd)}^m 	\\
		&\le\frac{1}{m-1}\|V_\varepsilon\|_{L^1}^m\|\rho_0\|_{L^m}^m=\frac{1}{m-1}\|V_1\|_{L^1}^m\|\rho_0\|_{L^m}^m<\infty.
	\end{align*}
\end{rem}

We now proceed with the JKO scheme associated to $\mf^\varepsilon$. First, we define a sequence recursively as follows:
\begin{itemize}
	\item fix a time step $\tau \in(0,1)$ such that $\rho_{\tau,\varepsilon}^0:=\rho_0$;
	\item for $n\in\mathbb{N}$ and given $\rhotne\in\mptrd$, choose
	\begin{equation}\label{eq:jko}
		\rhotnne\in\argmin_{\rho\in\mptrd}\left\{\frac{d_W^2(\rhotne,\rho)}{2\tau}+\mf^\varepsilon[\rho]\right\}.
	\end{equation}
\end{itemize}
The above sequence is well-defined for $\tau$ sufficiently small independently of $\varepsilon$ (given explicitly in~\Cref{lem:one-step-JKO-F}).

Let $T>0$ be fixed, and define a piecewise constant interpolation as follows: take $N:=\left[\frac{T}{\tau}\right]$ the largest integer less than or equal to $\frac{T}{\tau}$ and set
$$
\rhote(t)=\rhotne \qquad t\in((n-1)\tau,n\tau], \quad n=0,1,\dots,N,
$$
being $\rhotne$ defined in \eqref{eq:jko}. As usually proven, we derive energy and moments bounds sufficient to show narrow compactness. 

\begin{prop}[Narrow compactness, energy, $\&$ moments bound]\label{prop:en-ineq-mom-bound}
	Let $0 < \varepsilon_0<\infty$ be fixed and suppose $F$ satisfied~\ref{ass:AGS-F}. There exists an absolutely continuous curve $\tilde{\rho}^\varepsilon: [0,T]\rightarrow\mptrd$ such that the piecewise constant interpolation $\rhote$ admits a subsequence $\rho_{\tau_k}^\varepsilon$ narrowly converging to $\tilde{\rho}^\varepsilon$ uniformly in $t\in[0,T]$ and $0 < \varepsilon\le \varepsilon_0$ as $k\rightarrow +\infty$. Moreover, for any $t\in[0,T]$, the following uniform bounds in $\tau$ and $0< \varepsilon \le \varepsilon_0$ hold
	\begin{subequations}
		\begin{align*}
			\mf^\varepsilon[\tilde{\rho}^\varepsilon(t)]&\le\sup_{\varepsilon>0}\mf^\varepsilon[\rho_0],\quad
			m_2(\tilde{\rho}^\varepsilon)\le C\left(T, \, m_2(\rho_0), \, \varepsilon_0^2 m_2(V_1), \, \sup_{\varepsilon>0}\mf^\varepsilon[\rho_0]\right),
		\end{align*}
	\end{subequations}
where $C\left(T, \, m_2(\rho_0), \, \varepsilon_0^2m_2(V_1), \, \sup_{\varepsilon>0}\mf^\varepsilon[\rho_0]\right)>0$ is a uniform constant depending only on the quantities in the brackets.
\end{prop}
The following proof is based on~\cite{JKO98, AGS}.
\begin{proof}
	From the definition of the sequence $\{\rhotne\}_{n=0,\dots,N}$ it holds
	\begin{align}\label{eq:basic-ineq}
		\frac{d_W^2(\rhotne,\rhotnne)}{2\tau}+\mf^\varepsilon[\rhotnne]\le\mf^\varepsilon[\rhotne], \quad \forall n=0,\dots,N-1.
	\end{align}
	which implies $\mf^\varepsilon[\rhotnne]\le\mf^\varepsilon[\rhotne]$, and, in particular, the following bound for the regularised internal energy
	\begin{align}\label{eq:energy-ineq}
		\sup_{0 \le n \le N, \, N\tau \le T} \mf^\varepsilon[\rhotne]\le\mf^\varepsilon[\rho_0],
	\end{align}
	where the supremum is over all $n=0,\dots,N$ and $\tau\in(0,1)$ such that $N\tau \le T$ with $N := \left[\frac{T}{\tau}\right]$. By summing up over $k$ in inequality \eqref{eq:basic-ineq}, we obtain
	\begin{equation}
		\label{eq:sum-ineq}
		\sum_{k=m}^n\frac{d_W^2(\rho_{\tau,\varepsilon}^k,\rho_{\tau,\varepsilon}^{k+1})}{2\tau}\le \mf^\varepsilon[\rho_{\tau,\varepsilon}^{m}]-\mf^\varepsilon[\rhotnne], \quad \forall 0 \le m \le n \le N-1.
	\end{equation}
	\underline{Bounded second moment:} We claim the existence of some uniform constant $C>0$ (depending on the quantities discussed in the statement of this result) such that
	\begin{equation}
		\label{eq:mom-inequality}
		\sup_{0 \le n \le N, \, N\tau \le T}m_2(\rho_{\tau,\varepsilon}^n) \le C.
	\end{equation}
	By~\Cref{rem:mom-ineq}, for any fixed $n=0,\dots,N-1$, we begin with
	\[
	m_2(\rho_{\tau, \varepsilon}^{n+1})\le 2d_W^2(\rho_0,\rho_{\tau,\varepsilon}^{n+1}) + 2m_2(\rho_0).
	\] 
	We use the triangle inequality and Cauchy-Schwarz to estimate the $d_W^2$ term
	\[
	m_2(\rho_{\tau,\varepsilon}^{n+1}) \le 2(n+1) \sum_{k=0}^{n}d_W^2(\rho_{\tau,\varepsilon}^k, \, \rho_{\tau, \varepsilon}^{k+1}) + 2m_2(\rho_0).
	\]
	We replace the summation with~\eqref{eq:sum-ineq} and use $n+1 \le N$ to obtain
	\[
	m_2(\rho_{\tau, \varepsilon}^{n+1})\le 4 T (\mf^\varepsilon[\rho_0] - \mf^\varepsilon[\rho_{\tau,\varepsilon}^{n+1}]) + 2m_2(\rho_0).
	\]
	We insert the lower bound for $\mf^\varepsilon$ from~\eqref{eq:low-bdd-fveps} so that we have
	\[
	m_2(\rho_{\tau,\varepsilon}^{n+1}) \le 4T \left(
	\mf^\varepsilon[\rho_0] +c_1 + c_2 C_{d,\alpha}(1 + \varepsilon^2 m_2(V_1) + m_2(\rho_{\tau,\varepsilon}^{n+1}))^\alpha
	\right) + 2m_2(\rho_0).
	\]
	Keeping in mind that we can assume $\alpha < 1$ without loss of generality, this final inequality implies the bound~\eqref{eq:mom-inequality}. This can be seen by analysing sequences $x_n\ge 0$ satisfying $x_n \le C_1 + C_2x_n^\alpha$.
	
	\underline{Bounded squared 2-Wasserstein distance:} We claim the existence of some uniform constant $c>0$ such that
	\begin{equation}
		\label{eq:total-square-en-estim}
		\sum_{k=m}^nd_W^2(\rho_{\tau,\varepsilon}^k, \, \rho_{\tau,\varepsilon}^{k+1}) \le c\tau, \quad \forall 0 \le m \le n \le N-1.
	\end{equation}
	We insert the upper bound of $\mf^\varepsilon$~\eqref{eq:energy-ineq} and the lower bound of $\mf^\varepsilon$~\eqref{eq:low-bdd-fveps} into~\eqref{eq:sum-ineq} to obtain
	\begin{align*}
		\sum_{k=m}^nd_W^2(\rho_{\tau,\varepsilon}^k, \, \rho_{\tau,\varepsilon}^{k+1}) &\le 2\tau (\mf^\varepsilon[\rho_{\tau,\varepsilon}^{m}]-\mf^\varepsilon[\rhotnne]) \\
		&\le 2\tau (\mf^\varepsilon[\rho_0] +c_1 + c_2 C_{d,\alpha}(1 + \varepsilon^2 m_2(V_1) + m_2(\rho_{\tau,\varepsilon}^{n+1}))^\alpha).
	\end{align*}
	By the uniform second moment estimate~\eqref{eq:mom-inequality}, the inequality~\eqref{eq:total-square-en-estim} is verified.
	
	\underline{Compactness:} Now, let us consider $0< s<t$ such that $s\in((m-1)\tau,m\tau]$ and $t\in((n-1)\tau,n\tau]$ (which implies $|n-m|<\frac{|t-s|}{\tau}+1$); by Cauchy-Schwarz inequality and \eqref{eq:total-square-en-estim}, we obtain
	\begin{equation}\label{eq:holder-cont}
		\begin{split}
			d_W(\rhote(s),\rhote(t))&\le\sum_{k=m}^{n-1}d_W(\rho_{\tau,\varepsilon}^{k},\rho_{\tau,\varepsilon}^{k+1})\le\left(\sum_{k=m}^{n-1}d_W^2(\rho_{\tau,\varepsilon}^{k},\rho_{\tau,\varepsilon}^{k+1})\right)^{\frac{1}{2}}|n-m|^{\frac{1}{2}}\\&\le c \left(\sqrt{|t-s|}+\sqrt{\tau}\right),
		\end{split}
	\end{equation}
	where $c$ is a positive constant.
	Thus $\rhote$ is $\frac{1}{2}$-H\"{o}lder equicontinuous, up to a negligible error of order $\sqrt{\tau}$. By using a refined version of Ascoli-Arzel\`{a}'s theorem, \cite[Proposition 3.3.1]{AGS}, we obtain $\rhote$ admits a subsequence narrowly converging to a limit $\tilde{\rho}^\varepsilon$ as $\tau\to0^+$ uniformly on $[0,T]$.
	Since $|\cdot|^2$ is lower semicontinuous and bounded from below, we actually have for any $t\in[0,T]$
	\begin{align*}
		\liminf_{k\to+\infty}\int_\Rd|x|^2\,d\rho_{\tau_k}^\varepsilon(x)\ge\int_\Rd|x|^2\,d\tilde{\rho}^\varepsilon(x).
	\end{align*}
	Moreover, $\mf^\varepsilon$ is lower semicontinuous and bounded from below since $V_\varepsilon*\rho$ is bounded. Then an application of Fatou's lemma implies 
	\begin{align*}
		\liminf_{k\to+\infty}\mf^\varepsilon[\rho_{\tau_k}^\varepsilon]&\ge\mf^\varepsilon[\tilde{\rho}^\varepsilon],
	\end{align*}
	whence the thesis follows by applying the above inequalities to \eqref{eq:energy-ineq} and \eqref{eq:mom-inequality}.
\end{proof}

Next, we show that $\tilde{\rho}^\varepsilon$ provided by~\Cref{prop:en-ineq-mom-bound} is indeed a solution to~\eqref{eq:nlie-class}, thus proving~\Cref{thm:exist_nlie-class}. Since we make use of \ref{ass:diff-F}, the theorem below does not include linear diffusion corresponding to $F(x) = x\log x$.

\begin{proof}[Proof of~\Cref{thm:exist_nlie-class}]
	Let us consider two consecutive elements of the sequence $\{\rhotne\}_{n\in\mathbb{N}}$ defined from the JKO step \eqref{eq:jko}, i.e. $\rhotne$ and $\rhotnne$. We perturb $\rhotnne$ by using the map $P^\sigma=\mathrm{id}+\sigma\zeta$, for some $\zeta\in C_c^\infty(\Rd;\Rd)$ and $\sigma>0$, that is we consider the perturbation
	\begin{equation}\label{eq:perturbation}
		\rho^\sigma:= P_\#^\sigma\rhotnne.
	\end{equation}
	Being $\rhotnne$ a minimiser of \eqref{eq:jko}, we have 
	\begin{equation}\label{eq:optimality}
		\frac{1}{2\tau}\left[\frac{d_W^2(\rhotne,\rho^\sigma)- d_W^2(\rhotne, \rhotnne)}{\sigma}\right]+\frac{\mf^\varepsilon[\rho^\sigma]-\mf^\varepsilon[\rhotnne]}{\sigma}\ge0.
	\end{equation}
We now let $\sigma\to0$ in \eqref{eq:optimality} analysing the two terms involved separately.

\underline{The energy functional terms in~\eqref{eq:optimality}:} In this part of the proof, we aim to show
\begin{equation}
\label{eq:energy-diff}
\frac{\mf^\varepsilon[\rho^\sigma] - \mf^\varepsilon[\rho_{\tau,\varepsilon}^{n+1}]}{\sigma} \to \int_{\Rd}\zeta(x)\cdot \nabla V_\varepsilon * [F'(V_\varepsilon * \rho_{\tau,\varepsilon}^{n+1})](x) \, d\rho_{\tau,\varepsilon}^{n+1}(x), \quad \sigma\to 0.
\end{equation}
We apply the mean-value form of the Taylor expansion to $F$
\begin{align}
	\label{eq:interaction-terms}
\begin{split}
&\quad \frac{1}{\sigma}\int_{\Rd}(F(V_\varepsilon * \rho^\sigma(x)) - F(V_\varepsilon * \rho_{\tau,\varepsilon}^{n+1}))dx 	\\
&= \frac{1}{\sigma}\int_{\Rd}(V_\varepsilon*\rho^\sigma(x) - V_\varepsilon* \rho_{\tau,\varepsilon}^{n+1}(x))\underbrace{\int_0^1 F'\left(
tV_\varepsilon*\rho^\sigma(x) + (1-t)V_\varepsilon * \rho_{\tau,\varepsilon}^{n+1}(x)
\right)dt}_{=:M_\varepsilon^\sigma(x)}\, dx 	\\
&= \frac{1}{\sigma}\int_{\Rd} (V_\varepsilon * M_\varepsilon^\sigma)(x) d[\rho^\sigma - \rho_{\tau,\varepsilon}^{n+1}](x) = \int_{\Rd}\frac{(V_\varepsilon*M_\varepsilon^\sigma)(P^\sigma(x)) - (V_\varepsilon*M_\varepsilon^\sigma)(x)}{\sigma}d\rho_{\tau,\varepsilon}^{n+1}(x) 	\\
&= \int_{\Rd}\left\{\int_{\Rd} \left(\frac{V_\varepsilon(P^\sigma(x) - y) - V_\varepsilon(x-y)}{\sigma}\right) M_\varepsilon^\sigma(y)dy\right\} d\rho_{\tau,\varepsilon}^{n+1}(x).
\end{split}
\end{align}
In the last few lines, we used the definition of $\rho^\sigma$ from~\eqref{eq:perturbation} and expanded the convolution. The limit \eqref{eq:energy-diff} is achieved by first proving
\begin{align}
\label{eq:aex-convergence}
\begin{split}
\int_{\Rd}&\left(\frac{V_\varepsilon(P^\sigma(x) - y) - V_\varepsilon(x-y)}{\sigma}\right) M_\varepsilon^\sigma(y)dy 	\\
&\to \zeta(x)\cdot \int_{\Rd}\nabla V_\varepsilon(x-y) F'(V_\varepsilon * \rho_{\tau,\varepsilon}^{n+1}(y))dy, \quad \sigma \to 0, \, \rho_{\tau,\varepsilon}^{n+1}\text{-almost every }x\in\Rd.
\end{split}
\end{align}
This is exactly $\zeta(x)\cdot \nabla V_\varepsilon * [F'(V_{\varepsilon}*\rho_{\tau,\varepsilon}^{n+1})](x)$ which appears as the integrand in~\eqref{eq:energy-diff}. Assuming this is true for now, by Egorov's theorem, for every $\eta>0$, there exists a measurable set $S_\eta\subset \Rd$ such that $\rho_{\tau,\varepsilon}^{n+1}(S_\eta) < \eta$ and the convergence~\eqref{eq:aex-convergence} is uniform on $\Rd\setminus S_\eta$. Continuing from the last line of~\eqref{eq:interaction-terms}, we have
\begin{align}
	\label{eq:egorov}
	\begin{split}
\frac{\mf^\varepsilon[\rho^\sigma] - \mf^\varepsilon[\rho_{\tau,\varepsilon}^{n+1}]}{\sigma} = &\int_{S_\eta} \left\{\int_{\Rd} \left(\frac{V_\varepsilon(P^\sigma(x) - y) - V_\varepsilon(x-y)}{\sigma}\right) M_\varepsilon^\sigma(y)dy\right\} d\rho_{\tau,\varepsilon}^{n+1}(x) 	\\
+ &\int_{\Rd\setminus S_\eta} \left\{\int_{\Rd} \left(\frac{V_\varepsilon(P^\sigma(x) - y) - V_\varepsilon(x-y)}{\sigma}\right) M_\varepsilon^\sigma(y)dy\right\} d\rho_{\tau,\varepsilon}^{n+1}(x).
\end{split}
\end{align}
The integral over $\Rd\setminus S_\eta$ passes well in the limit $\sigma\to 0$ owing to~\eqref{eq:aex-convergence} and Egorov's theorem, so~\eqref{eq:energy-diff} is achieved once we show that the integral over $S_\eta$ is small. We apply the mean-value form of Taylor's theorem for $V_\varepsilon$ and~\Cref{lem:bdd-comp-F-V} (with $\rho = t\rho^\sigma + (1-t)\rho_{\tau,\varepsilon}^{n+1}$ and $C=1$) to estimate $M_\varepsilon^\sigma$ and obtain
\begin{align*}
&\left|\int_{\Rd} \left(\frac{V_\varepsilon(P^\sigma(x) - y) - V_\varepsilon(x-y)}{\sigma}\right) M_\varepsilon^\sigma(y)dy\right|\\
&\le \|F'\|_{L^\infty([0,\,\|V_\varepsilon\|_{L^\infty}])} |\zeta(x)|\int_{\Rd} \int_0^1 |\nabla V_\varepsilon (x + s\sigma \zeta(x) - y)|ds dy 	\\
&= \|F'\|_{L^\infty([0,\,\|V_\varepsilon\|_{L^\infty}])}|\zeta(x)|\int_0^1 \int_{\Rd} |\nabla V_\varepsilon(x + s \sigma \zeta(x) - y)|dy ds\\
&= \|F'\|_{L^\infty([0,\,\|V_\varepsilon\|_{L^\infty}])}|\zeta(x)|\int_0^1 \int_{\Rd} |\nabla V_\varepsilon(z)|dz ds 	\\
&= \|F'\|_{L^\infty([0,\,\|V_\varepsilon\|_{L^\infty}])}\|\nabla V_\varepsilon\|_{L^1}|\zeta(x)|.
\end{align*}
In the second to last line, we have used Fubini and the linear change of variables $z = x + s\sigma\zeta(x) - y$ for fixed $x$. Therefore, the integral over $S_\eta$ from~\eqref{eq:egorov} can be estimated by
\[
\left|\int_{S_\eta}\! \left\{\int_{\Rd} \left(\!\frac{V_\varepsilon(P^\sigma(x) - y) - V_\varepsilon(x-y)}{\sigma}\right) \!M_\varepsilon^\sigma(y)dy\right\} \!d\rho_{\tau,\varepsilon}^{n+1}(x) \right|\! \le \!\|\zeta\|_{L^\infty} \|F'\|_{L^\infty([0,\,\|V_\varepsilon\|_{L^\infty}])} \|\nabla V_\varepsilon\|_{L^1} \, \eta,
\]
which is negligible by taking $\eta \to 0$.

\underline{Proving~\eqref{eq:aex-convergence}:} Throughout this step, we fix $x\in\Rd$. We again use the mean-value form of Taylor's theorem to rewrite the difference quotient appearing in~\eqref{eq:aex-convergence}
\[
\int_{\Rd} \left(\frac{V_\varepsilon(P^\sigma(x) - y) - V_\varepsilon(x-y)}{\sigma}\right) M_\varepsilon^\sigma(y)dy = \zeta(x)\cdot \int_{\Rd}\left(\int_0^1\nabla V_\varepsilon(x + s\sigma\zeta(x) - y)\,ds\right) M_\varepsilon^\sigma(y)dy.
\]
We majorise the integrand with the sequence
\[
\left|
\int_0^1\nabla V_\varepsilon(x + s\sigma\zeta(x) - y)\,ds \, M_\varepsilon^\sigma(y)
\right| \le \|F'\|_{L^\infty([0,\|V_\varepsilon\|_{L^\infty}])}\int_0^1|\nabla V_\varepsilon(x + s\sigma\zeta(x) - y)|ds.
\]
We seek to apply~\Cref{thm:EDCT} on $X = \Rd$ with
\begin{align*}
	f^\sigma(y) &:= \int_0^1\nabla V_\varepsilon(x + s\sigma\zeta(x) - y)\,ds \, M_\varepsilon^\sigma(y), \\
	g^\sigma(y) 	&:= \|F'\|_{L^\infty([0,\|V_\varepsilon\|_{L^\infty}])}\int_0^1|\nabla V_\varepsilon(x + s\sigma\zeta(x) - y)|ds.
\end{align*}
We have already shown the majorisation $|f^\sigma(y)|\le g^\sigma(y)$ and the convergence
\begin{align*}
f^\sigma(y) \to f(y) &:= \nabla V_\varepsilon (x-y) F'(V_\varepsilon*\rho(y)), 	\\
g^\sigma(y) \to g(y) &:= \|F'\|_{L^\infty([0,\|V_\varepsilon\|_{L^\infty}])} |\nabla V_\varepsilon(x-y)|,
\end{align*}
for almost every $y\in\Rd$ can be proven using the usual Dominated Convergence Theorem. In particular, the growth estimate $|\nabla V_1(z)| \le C(1+|z|)$ treats the integration $\int_0^1\, ds$. On the other hand, for $M_\varepsilon^\sigma(y)$, the composition $F'(tV_\varepsilon*\rho^\sigma(y) + (1-t)V_\varepsilon*\rho_{\tau,\varepsilon}^{n+1}(y))$ is bounded uniformly in $\sigma$ by~\Cref{lem:bdd-comp-F-V}. We verify the last assumption of~\Cref{thm:EDCT} using Fubini and the change of variables $z = -s\sigma\zeta(x) + y$.
\begin{align*}
\quad \int_{\Rd}g^\sigma(y) dy &= \|F'\|_{L^\infty([0,\|V_\varepsilon\|_{L^\infty}])} \int_{\Rd} \int_0^1|\nabla V_\varepsilon(x + s\sigma\zeta(x) - y)|ds dy \\
&=\|F'\|_{L^\infty([0,\|V_\varepsilon\|_{L^\infty}])} \int_0^1 \int_{\Rd}|\nabla V_\varepsilon(x + s\sigma\zeta(x) - y)|dy \, ds \\
&= \|F'\|_{L^\infty([0,\|V_\varepsilon\|_{L^\infty}])} \int_0^1 \int_{\Rd}|\nabla V_\varepsilon(x-z)|dz \, ds 	\\
&= \int_{\Rd}g(y) dy.
\end{align*}
Therefore, we can apply~\Cref{thm:EDCT} and~\eqref{eq:aex-convergence} is established.

\underline{The 2-Wasserstein terms in~\eqref{eq:optimality}:} the treatment here is standard and we reproduce the proof in~\cite[Theorem 3.1]{BE22} for completeness. Consider an optimal transport plan $\gamma_{\tau,\varepsilon}^{n+1} \in \Gamma_o(\rho_{\tau,\varepsilon}^n, \, \rho_{\tau,\varepsilon}^{n+1})$ between $\rho_{\tau,\varepsilon}^n$ and $\rho_{\tau,\varepsilon}^{n+1}$. By definition of $d_W$, we have
\begin{equation*}
	\begin{split}
		\frac{1}{2\tau}\left[\frac{d_W^2(\rhotne, \rho^\sigma)-d_W^2(\rhotne, \rhotnne)}{\sigma}\right]&\le\frac{1}{2\tau\sigma}\iint_\Rdd\left(|x-P^\sigma(y)|^2 -|x-y|^2\right)\,d\gamma_{\tau,\varepsilon}^n(x,y)\\
		&=\frac{1}{2\tau\sigma}\iint_\Rdd\left(|x-y-\sigma\zeta(y)|^2 -|x-y|^2\right)\,d\gamma_{\tau,\varepsilon}^n(x,y)\\
		&=-\frac{1}{\tau}\iint_\Rdd(x-y)\cdot \zeta(y)\,d\gamma_{\tau,\varepsilon}^n(x,y)+o(\sigma),
	\end{split}
\end{equation*}
where in the last equality we applied a first order Taylor expansion. By sending $\sigma$ to $0$ and recalling~\eqref{eq:optimality}, it holds
\begin{equation*}
	\frac{1}{\tau}\iint_\Rdd(x-y)\cdot \zeta(y)\,d\gamma_{\tau,\varepsilon}^n(x,y)\le \int_{\Rd}\zeta(x)\cdot \nabla V_\varepsilon*[F'(V_\varepsilon*\rhotnne)](x) d\rhotnne(x).
\end{equation*}
Repeating the same computation for $\sigma\le0$, we actually obtain an equality, that is, for $\zeta=\nabla\varphi$
\begin{equation}\label{eq:sol-discreta-sigma}
	\begin{split}
		\frac{1}{\tau}\!\iint_\Rdd(x-y)\cdot \nabla\varphi(y)d\gamma_{\tau,\varepsilon}^n(x,y)\!=\!\int_{\Rd}\nabla \varphi(x)\cdot \nabla V_\varepsilon*[F'(V_\varepsilon*\rhotnne)](x) d\rhotnne(x).
	\end{split}
\end{equation}
Note that the H\"older estimate \eqref{eq:holder-cont} and $(x-y)\cdot \nabla\varphi(y)=\varphi(x)-\varphi(y)+o(|x-y|^2)$ imply
\[
\frac{1}{\tau}\iint_\Rdd(x-y)\cdot \nabla\varphi(y)\,d\gamma_{\tau,\varepsilon}^n(x,y)=\frac{1}{\tau}\int_\Rd\varphi(x)\,d(\rhotne-\rhotnne)(x) + O(\tau).
\]
Now, let $0\le s<t$ be fixed, with
$$
h=\left[\frac{s}{\tau}\right]+1\quad \text{and}\quad k=\left[\frac{t}{\tau}\right].
$$
Taking into account the last equality, by summing in \eqref{eq:sol-discreta-sigma} over $j$ from $h$ to $k$, we obtain
\begin{align*}
	\int_\Rd\varphi(x)\,d\rho_{\tau,\varepsilon}^{k+1}-&\int_\Rd\varphi(x)\,d\rho_{\tau,\varepsilon}^h+O(\tau^2)=\\
	&-\tau\sum_{j=h}^k \int_{\Rd}\nabla \varphi(x)\cdot \nabla V_\varepsilon * [F'(V_\varepsilon*\rho_{\tau,\varepsilon}^{j+1})](x)d\rho_{\tau,\varepsilon}^{j+1}(x),
\end{align*}
which is equivalent to
\begin{align}
\label{eq:weak-form-tau}
\begin{split}
	\int_\Rd\varphi(x)\,d\rhote(t)(x)-&\int_\Rd\varphi(x)\,d\rhote(s)(x)+O(\tau^2)=\\
	&-\int_s^t\int_{\Rd}\nabla\varphi(x)\cdot \nabla V_\varepsilon* [F'(V_\varepsilon * \rho_\tau^\varepsilon(r))](x)\,d\rhote(r)(x)\,dr.
\end{split}
\end{align}

It remains to pass the limit $\tau\downarrow 0$ up to a subsequence for $\rho_\tau^\varepsilon \rightharpoonup \tilde{\rho}^\varepsilon$ as in~\Cref{prop:en-ineq-mom-bound}. More specifically, the result there states that $\rho_\tau^\varepsilon$ narrowly converges uniformly in $t\in[0,T]$ to $\tilde{\rho}^\varepsilon$ as (a subsequence of) $\tau\downarrow 0$. Clearly, the left-hand side of~\eqref{eq:weak-form-tau} passes easily in the limit $\tau\downarrow 0$ so we only focus on the right-hand side. Let us take the following statement for granted: for fixed $\varepsilon>0$ and almost every $r\in [0,T]$, we have
\begin{align}
	\label{eq:AEC}
	\begin{split}
	& \int_{\Rd}\!\!\nabla\varphi(x)\!\cdot\! \nabla V_\varepsilon* [F'(V_\varepsilon * \rho_\tau^\varepsilon(r))](x)\,d\rhote(r)(x) \\
	&\qquad \qquad  \to \int_{\Rd}\!\!\nabla\varphi(x)\!\cdot \!\nabla V_\varepsilon* [F'(V_\varepsilon * \tilde{\rho}^\varepsilon(r))](x)\,d\tilde{\rho}^\varepsilon(r)(x), \quad \tau\downarrow 0, \, \text{almost every }r\in[0,T].
	\end{split}
\end{align}
Passing to the limit $\tau\downarrow 0$ on the right-hand side of~\eqref{eq:weak-form-tau} reduces to finding an $L^1((s,t);dr)$ majorant, assuming~\eqref{eq:AEC} holds. By Young's convolution inequality and~\Cref{lem:bdd-comp-F-V}, we have
\[
|\nabla V_\varepsilon * [F'(V_\varepsilon*\rho_\tau^\varepsilon(r))](x)| \le \|\nabla V_\varepsilon\|_{L^1}\|F'\|_{L^\infty([0, \, \|V_\varepsilon\|_{L^\infty}])}.
\]
Overall, this implies the uniform estimate in $\tau>0$
\begin{align*}
&\quad \left|
\int_{\Rd}\nabla\varphi(x)\cdot \nabla V_\varepsilon* [F'(V_\varepsilon * \rho_\tau^\varepsilon(r))](x)\,d\rhote(r)(x)
\right| \le \|\nabla \varphi\|_{L^\infty}\|\nabla V_\varepsilon\|_{L^1}\|F'\|_{L^\infty([0, \, \|V_\varepsilon\|_{L^\infty}])}.
\end{align*}
Hence, we can pass to the limit $\tau\downarrow 0$ in the right-hand side of~\eqref{eq:weak-form-tau} and conclude.

Let us prove~\eqref{eq:AEC}. We fix $r\in[0,T]$ and henceforth drop the explicit dependence on this variable. We add and subtract
\begin{align}
&\quad \int_{\Rd}\nabla\varphi(x)\cdot \nabla V_\varepsilon* [F'(V_\varepsilon * \rho_\tau^\varepsilon)](x)\,d\rhote(x) - \int_{\Rd}\nabla\varphi(x)\cdot \nabla V_\varepsilon* [F'(V_\varepsilon * \tilde{\rho}^\varepsilon)](x)\,d\tilde{\rho}^\varepsilon(x) \notag\\
&=\int_{\supp \varphi}\nabla\varphi(x)\cdot \nabla V_\varepsilon* [F'(V_\varepsilon * \rho_\tau^\varepsilon) - F'(V_\varepsilon*\tilde{\rho}^\varepsilon)](x)\,d\rhote(x) 	\label{eq:diff-F'} 	\\
&\qquad \qquad \qquad \qquad \qquad \qquad \qquad + \int_{\supp \varphi}\nabla\varphi(x)\cdot \nabla V_\varepsilon* [F'(V_\varepsilon * \tilde{\rho}^\varepsilon)](x)\,d[\rho_\tau^\varepsilon -\tilde{\rho}^\varepsilon](x). 	\label{eq:diff-rho}
\end{align}
Fix small $\eta>0$ and find $R>1$ large enough such that $\int_{\Rd\setminus B_R}|\nabla V_\varepsilon(y)| dy < \eta$ where $B_R$ denotes the open ball of radius $R$ centred at the origin. We begin with the difference in~\eqref{eq:diff-F'} by expanding the convolution
\begin{align*}
	&\quad \left|\nabla V_\varepsilon* [F'(V_\varepsilon * \rho_\tau^\varepsilon) - F'(V_\varepsilon*\tilde{\rho}^\varepsilon)](x)\right| \le \int_{\Rd} |\nabla V_\varepsilon(y)| | F'(V_\varepsilon * \rho_\tau^\varepsilon(x-y)) - F'(V_\varepsilon*\tilde{\rho}^\varepsilon(x-y)) |dy.
\end{align*}
Up to a further subsequence, \Cref{cor:comp-f-conv} gives
\[
\sup_{x\in \supp\varphi, \, y \in \bar{B}_R} | F'(V_\varepsilon * \rho_\tau^\varepsilon(x-y)) - F'(V_\varepsilon*\tilde{\rho}^\varepsilon(x-y)) | < \eta,
\]
for $\tau>0$ sufficiently small. Hence, 
\begin{equation}
	\label{eq:near-field}
\int_{B_R} |\nabla V_\varepsilon(y)| | F'(V_\varepsilon * \rho_\tau^\varepsilon(x-y)) - F'(V_\varepsilon*\tilde{\rho}^\varepsilon(x-y)) |dy < \|\nabla V_\varepsilon\|_{L^1}\eta.
\end{equation}
Concerning the integral over $\Rd\setminus B_R$, we apply~\Cref{lem:bdd-comp-F-V} to obtain (uniformly in $\tau>0$)
\begin{equation}
	\label{eq:far-field}
\int_{\Rd\setminus B_R} |\nabla V_\varepsilon(y)| | F'(V_\varepsilon * \rho_\tau^\varepsilon(x-y)) - F'(V_\varepsilon*\tilde{\rho}^\varepsilon(x-y)) |dy < 2\|F'\|_{L^\infty([0, \, \|V_\varepsilon\|_{L^\infty}])} \eta.
\end{equation}
These inequalities imply that the integral in~\eqref{eq:diff-F'} can be made arbitrarily small in the limit $\tau\downarrow 0$.

Turning to the difference in~\eqref{eq:diff-rho}, we only need to show that $\nabla V_\varepsilon * [F'(V_\varepsilon*\tilde{\rho}^\varepsilon)]$ is continuous on $\supp \varphi$. Then, we can appeal to the narrow convergence $\rho_\tau^\varepsilon \rightharpoonup \tilde{\rho}^\varepsilon$ in duality with continuous and bounded functions. Suppose $x^n\in \supp\varphi$ is a sequence which converges to $x\in \supp\varphi$, we compare the difference
\begin{align*}
&\quad \nabla V_\varepsilon*F'(V_\varepsilon*\tilde{\rho}^\varepsilon)(x^n) - \nabla V_\varepsilon*F'(V_\varepsilon*\tilde{\rho}^\varepsilon)(x) \\
&= \int_{B_R} \nabla V_\varepsilon(y) [F'(V_\varepsilon*\tilde{\rho}^\varepsilon(x^n-y)) - F'(V_\varepsilon*\tilde{\rho}^\varepsilon(x-y))]dy 	\\
&\qquad \qquad \qquad \qquad \qquad \qquad \qquad  + \int_{\Rd\setminus B_R} \nabla V_\varepsilon(y) [F'(V_\varepsilon*\tilde{\rho}^\varepsilon(x^n-y)) - F'(V_\varepsilon*\tilde{\rho}^\varepsilon(x-y))]dy.
\end{align*}
The integral over $B_R$ can be made arbitrarily small as $n\to \infty$ owing to the uniform continuity of $F'(V_\varepsilon*\tilde{\rho}^\varepsilon(\cdot))$ from~\Cref{cor:comp-f-conv} and integrability of $\nabla V_\varepsilon$. This is similar to what is done for~\eqref{eq:near-field}. The other integral over $\Rd\setminus B_R$ can be made arbitrarily small by the same argument for~\eqref{eq:far-field}.
\end{proof}

\section{Compactness in the limit $\varepsilon\downarrow 0$}
\label{sec:compact}

This section discusses the construction of a limit $\rho$ for a subsequence of $\{\tilde{\rho}^\varepsilon\}_{\varepsilon>0}$. The key estimate is~\Cref{lem:h1-bound} which we are able to prove for general functions $F$ satisfying~\ref{ass:AGS-F}, \ref{ass:diff-reg}, and the growth conditions \ref{ass:F-m}.

Assumptions~\ref{ass:AGS-F}, \ref{ass:diff-reg}, and \ref{ass:F-m} cover all the power laws $F(x) = \frac{1}{m-1}|x|^m$ for $m>1$ and $F(x) = x\log x$ (corresponding to $m=1$). In the case $m>1$, \ref{ass:F-m} implies~\ref{ass:diff-F} since $F'$ can be extended to $x=0$. More precisely, if $F''$ satisfies the bounds in~\ref{ass:F-m} for some $m>1$, then $F''$ is locally integrable around 0. By the fundamental theorem of Calculus, 
\[
F'(x) = F'(1) - \int_x^1 F''(t)dt, \quad \forall x>0.
\]
Owing to Lebesgue's dominated convergence theorem, the right-hand side has a limit as $x\downarrow 0$ and therefore so does the left-hand side which we call $F'(0) := \lim_{x\downarrow 0} F'(x)$.
\begin{rem}[Comments on~\ref{ass:F-m}]
\label{rem:F-m}
Combining~\ref{ass:F-m} with the assumption $F(0)=0$ from~\ref{ass:AGS-F} gives, for $m>1$,
\[
\frac{c_1}{m(m-1)}x^m \le F(x) - F'(0)x \le \frac{c_2}{m(m-1)}x^m.
\]
The inequalities above and the uniform bound for $\mf^\varepsilon[\tilde{\rho}^\varepsilon(t)]$ from~\Cref{prop:en-ineq-mom-bound} yield the following integrability estimate uniform in $t\in[0,T]$ and $\varepsilon>0$
\begin{align}
	\label{eq:lmest}
 \begin{split}
&\quad \|V_\varepsilon * \tilde{\rho}^\varepsilon(t) \|_{L^m(\R^d)}^m \le \frac{m(m-1)}{c_1}\mf^\varepsilon[\tilde{\rho}^\varepsilon(t)] - \frac{m(m-1)}{c_1}F'(0)  \\
&\le \frac{m(m-1)}{c_1}\mf^\varepsilon[\rho_0]- \frac{m(m-1)}{c_1}F'(0) \le \frac{c_2}{c_1}\|V_\varepsilon*\rho_0\|_{L^m}^m \le \frac{c_2}{c_1}\|\rho_0\|_{L^m}^m.
\end{split}
\end{align}
Concerning the $m=1$ case, we directly estimate
\begin{equation}
    \label{eq:entest}
    \mathcal{H}[V_\varepsilon*\tilde{\rho}^\varepsilon(t)] = \mf^\varepsilon[\tilde{\rho}^\varepsilon(t)] \le \mf^\varepsilon[\rho_0] = \mathcal{H}[V_\varepsilon*\rho_0] \le \mathcal{H}[\rho_0].
\end{equation}
Here, we used Jensen's inequality with the convex function $x\log x$ and reference measure $V_\varepsilon$ to obtain $\mathcal{H}[V_\varepsilon*\rho_0]\le \int V_\varepsilon*(\rho_0\log \rho_0) = \int \rho\log \rho$ recalling $\int V_\varepsilon = 1$ as well as~\Cref{prop:en-ineq-mom-bound}.
\end{rem}
The sequence of solutions $\{\tilde{\rho}^\varepsilon\}_{\varepsilon>0}$ to~\eqref{eq:nlie-class} constructed in~\Cref{sec:nlie} is the candidate approximating \textit{weak solution} of~\eqref{eq:nonlinear-diffusion}. As $\{\tilde{\rho}^\varepsilon\}_{\varepsilon>0}$ is in general a sequence of measures, it is useful to consider the regularised version, $V_\varepsilon * \tilde{\rho}^\varepsilon$. For brevity, we drop the tilde on $\rho^\varepsilon$ from now on. First, we state compactness of $\{\rho^\varepsilon\}_{\varepsilon>0}$ in $C([0,T]; \mP_2(\R^d))$.
\begin{prop}
\label{prop:limit-rho}
There exists an absolutely continuous curve $\tilde{\rho}:[0,T]\to\mptrd$ such that the sequence $\{\rho^\varepsilon\}_{\varepsilon>0}$ admits a subsequence $\{\rho^{\varepsilon_k}\}$ such that $\rho^{\varepsilon_k}(t)$ narrow converges to $\tilde{\rho}(t)$ for any $t\in[0,T]$ as $k\to+\infty$.
\end{prop}
\begin{proof}
The proof is exactly the same as in~\cite[Proposition 4.1]{BE22} using a refined version of Ascoli-Arzel\`a~\cite[Proposition 3.3.1]{AGS}.
\end{proof}
The narrow convergence proven in~\Cref{prop:limit-rho} is not sufficient to pass to the limit $\varepsilon\downarrow 0$ from~\eqref{eq:nlie-class} to~\eqref{eq:nonlinear-diffusion}. For this reason, we study the sequence $v^\varepsilon(t) := V_\varepsilon * \rho^\varepsilon(t)$ for $t\in[0,T]$ (we drop the subscript $k$ for simplicity). We obtain higher regularity estimates uniform in $\varepsilon$ by using the \textit{flow interchange technique} developed by Matthes, McCann, and Savar\'e in~\cite{MMCS}. The strategy is to compute the dissipation of $\mf^\varepsilon$ along a solution of an \textit{auxiliary gradient flow}. This flow is chosen so that it satisfies an \textit{Evolution Variational Inequality (EVI)} which allows us to obtain the desired estimate leading to compactness.

Since the seminal work of Jordan, Kinderlehrer, and Otto~\cite{JKO98}, it is known that the heat equation can be interpreted as the 2-Wasserstein gradient flow of the Boltzmann entropy $\mathcal{H}$ (see below for the precise definition). Moreover the heat semigroup, denoted by $S_\mathcal{H}$, is a 0-flow in the following sense.
\begin{defn}[$\lambda$-flow]
	A semigroup $S_{\mathcal{E}}:[0,+\infty]\times\mptrd\to\mptrd$ is a $\lambda$-flow for a functional $\mathcal{E}:\mptrd\to\R\cup\{+\infty\}$ with respect to the distance $d_W$ if, for an arbitrary $\rho\in\mptrd$, the curve $t\mapsto S_{\mathcal{E}}^t\rho$ is absolutely continuous on $[0,+\infty[$ and it satisfies the evolution variational inequality (EVI)
	\begin{equation}
		\frac{1}{2}\frac{d^+}{dt}d_W^2(S_{\mathcal{E}}^t\rho,\bar{\rho})+\frac{\lambda}{2}d_W^2(S_{\mathcal{E}}^t\rho,\bar{\rho})\le \mathcal{E}(\bar{\rho})-\mathcal{E}(S_{\mathcal{E}}^t\rho)
	\end{equation}
	for all $t\ge 0$, with respect to every reference measure $\bar{\rho}\in\mptrd$ such that $\mathcal{E}(\bar{\rho})<\infty$.
\end{defn}
Below we use the flow interchange by considering the heat equation as an auxiliary flow with respect to the Boltzmann entropy
\begin{equation}\label{eq:aux-func}
	\mh[\rho]=
	\begin{cases}
		\int_{\Rd}\rho(x)\log\rho(x)\,dx, &\rho \ll \text{Leb}(\Rd)\\
		+\infty, & \text{otherwise}
	\end{cases}.
\end{equation}
Again, when $\rho$ is an absolutely continuous measure with respect to Lebesgue, we identify its density as $\rho(x)$.
\begin{rem}
\label{rem:control_below_entropy}
We remind the reader that $\mathcal{H}[\rho]$ is bounded below by $m_2(\rho)$. This can be seen by looking at the \textit{relative entropy} with respect to the standard Gaussian on $\R^d$ denoted by $\mathcal{M}(x) = (2\pi)^{-d}\exp\{-|x|^2/2\}$. For any $\rho\in \mpdtard$, Jensen's inequality with the convex function $x\log x$ gives
\begin{align*}
\mathcal{H}[\rho \, |\, \mathcal{M}] &:= \int_{\R^d}\rho(x) \log \frac{\rho(x)}{\mathcal{M}(x)} dx = \int_{\R^d} \frac{\rho(x)}{\mathcal{M}(x)}\log \frac{\rho(x)}{\mathcal{M}(x)} \mathcal{M}(x)dx 	\\
&\ge \left(
\int_{\R^d} \frac{\rho(x)}{\mathcal{M}(x)} \mathcal{M}(x)dx
\right) \log \left(
\int_{\R^d} \frac{\rho(x)}{\mathcal{M}(x)} \mathcal{M}(x)dx
\right) = 0.
\end{align*}
This gives the lower bound for the entropy
\begin{align*}
\mathcal{H}[\rho] \ge \int_{\R^d}\rho(x)\log \mathcal{M}(x)dx = -d\log 2\pi - \frac{1}{2}m_2(\rho).
\end{align*}
\end{rem}
In the following, for any $\nu\in\mptrd$ such that $\mh(\nu)<+\infty$, we denote by $S_{\mh}^t\nu$ the solution at time $t$ of the heat equation coupled with an initial value $\nu$ at $t=0$. Moreover, for every $\rho\in\mptrd$, we define the dissipation of $\mf^\varepsilon$ along $S_{\mh}$ by
$$
D_{\mh}\mf^\varepsilon(\rho):=\limsup_{s\downarrow0}\left\{\frac{\mf^\varepsilon[\rho]-\mf^\varepsilon[S_{\mh}^s\rho]}{s}\right\}.
$$
In order to prove stronger compactness, we begin with an $L_t^2H_x^1$ estimate on the $\frac{m}{2}$ power of $v_\tau^\varepsilon = V_\varepsilon * \rho_\tau^\varepsilon$. This generalises Lemma 4.1 from~\cite{BE22}.
\begin{lem}
\label{lem:h1-bound}
Suppose $F$ satisfies~\ref{ass:AGS-F}, \ref{ass:diff-reg}, and~\ref{ass:F-m} for some $m\ge 1$. Let $\rho_0\in \mpdtard\cap L^m(\R^d)$. In the case $m=1$, assume further $\mathcal{H}[\rho_0] < +\infty$. Then, there exists a constant $C=C(\rho_0, \, V_1, \, T) >0$ such that
\[
\sup_{\varepsilon, \, \tau>0}\left\|
(v_\tau^\varepsilon)^\frac{m}{2}
\right\|_{L^2(0,T;\, H^1(\R^d))} \le C.
\]
\end{lem}
\begin{proof}
If $m=1$, then the $L_t^2L_x^2$ bound simply reads
\begin{align*}
	\left\| (v_\tau^\varepsilon)^\frac{1}{2}\right\|_{L^2(0,T;\, L^2(\R^d))}^2 = \int_0^T\int_{\R^d} V_\varepsilon*\rho_\tau^\varepsilon(t,x) \, dx dt = T,
\end{align*}
since both $\|V_\varepsilon\|_{L^1} = \int_{\R^d}d\rho_\tau^\varepsilon(t)(x) = 1$. For $m > 1$, the estimate is very similar to that of~\eqref{eq:lmest} applied to the pre-limit curves $\rhote$,
\begin{align*}
	\left\|(v_\tau^\varepsilon)^\frac{m}{2} \right\|_{L^2([0,T]; L^2(\R^d))}^2 &= \int_0^T\|V_\varepsilon*\rhote\|_{L^m}^m \,dt\le \frac{c_2 T}{c_1}\|\rho_0\|_{L^m}^m.
\end{align*}
The rest of this proof focuses on the uniform bound for $\nabla (v_\tau^\varepsilon)^\frac{m}{2}$. For $s>0$, we take $S_\mh^s\rho_{\tau, \varepsilon}^{n+1}$ as a competitor against $\rho_{\tau, \varepsilon}^{n+1}$ in the minimisation problem~\eqref{eq:jko}. We thus have
\[
\frac{1}{2\tau}d_W^2(\rhotnne,\rhotne)+\mf^\varepsilon[\rhotnne]\le\frac{1}{2\tau}d_W^2(S_{\mh}^s\rhotnne,\rhotne)+\mf^\varepsilon[S_{\mh}^s\rhotnne],
\]
which, dividing by $s>0$ and passing to $\limsup_{s\downarrow 0}$, gives
\begin{equation}\label{eq:flow-interchange}
	\tau D_{\mh}\mf^\varepsilon(\rhotnne)\le\left.\frac{1}{2}\frac{d^+}{dt}\right|_{t=0}\Big(d_W^2(S_{\mh}^t\rhotnne,\rhotne)\Big)\overset{\bm{(E.V.I.)}}{\le}\mh[\rhotne]-\mh[\rhotnne].
\end{equation}
In the last inequality we used that $S_\mh$ is a $0$-flow. Now, let us focus on the left hand side of~\eqref{eq:flow-interchange}. Firstly, note that
\begin{equation}\label{eq:integral-form-dis}
	\begin{split}
		D_{\mh}\mf^\varepsilon(\rhotnne)&=\limsup_{s\downarrow0}\left\{\frac{\mf^\varepsilon[\rhotnne]-\mf^\varepsilon[S_{\mh}^s\rhotnne]}{s}\right\}\\&=\limsup_{s\downarrow0}\int_0^1\left(-\frac{d}{dz}\Big|_{z=st}\mf^\varepsilon[S_{\mh}^{z}\rhotnne]\right)\,dt.
	\end{split}
\end{equation}
Thus, we now compute the time derivative inside the above integral. Using integration by parts, the $C^\infty$ regularity of the heat semigroup, and~\ref{ass:F-m}, we have
\begin{equation}\label{eq:deriv-dis}
	\begin{split}
		&\quad \frac{d}{dt}\mf^\varepsilon[S_{\mh}^t\rhotnne] =-\int_\Rd F''(V_\varepsilon* S_\mh^t\rhotnne)|\nabla V_\varepsilon* S_\mh^t\rhotnne|^2 dx \\
		&\le -c_1 \int_{\R^d} (V_\varepsilon* S_\mh^t\rhotnne)^{m-2}|\nabla V_\varepsilon* S_\mh^t\rhotnne|^2 dx =-\frac{4c_1}{m^2}\int_{\R^d}\left|
		\nabla (V_\varepsilon* S_\mh^t\rhotnne)^\frac{m}{2}
		\right|^2 dx.
	\end{split}
\end{equation}
The previous computation is justified since $S_{\mh}^t\rhotnne>0$ everywhere on $\Rd$ so there is no division by zero. By substituting~\eqref{eq:deriv-dis} into~\eqref{eq:integral-form-dis}, from~\eqref{eq:flow-interchange} we obtain
\[
\tau\liminf_{s\downarrow0}\int_0^1\int_{\Rd}\left|\nabla (V_\varepsilon*S_{\mh}^{st}\rhotnne)^\frac{m}{2}(x)\right|^2\,dx\,dt\le\frac{m^2}{4c_1}\left(\mh[\rhotne]-\mh[\rhotnne]\right).
\]
In order to pass to the limit $s\downarrow 0$ for $m>1$, we first deduce $V_\varepsilon * \rhotnne\in L^m$ by~\ref{ass:F-m} and~\Cref{prop:en-ineq-mom-bound}. Second, by standard properties of the heat semigroup, we obtain $V_\varepsilon * S_\mh^{st}\rhotnne \to V_\varepsilon * \rhotnne$ in $L^m$ as $s\downarrow 0$. Notice that the first and second steps are immediate for $m=1$. Third, by the inequality
\[
\left|
(V_\varepsilon*S_{\mh}^{st}\rhotnne)^\frac{m}{2} - (V_\varepsilon*\rhotnne)^\frac{m}{2}
\right|^2 \le 2\left(
(V_\varepsilon*S_{\mh}^{st}\rhotnne)^m + (V_\varepsilon*\rhotnne)^m
\right),
\]
we can apply~\Cref{thm:EDCT} to deduce $(V_\varepsilon*S_{\mh}^{st}\rhotnne)^\frac{m}{2} \to (V_\varepsilon*\rhotnne)^\frac{m}{2}$ in $L^2$ as $s\downarrow 0$. Finally, the weak $L^2$ lower semi-continuity of the $H^1$ semi-norm gives
\[
\tau\int_{\Rd}\left|\nabla |V_\varepsilon*\rhotnne|^\frac{m}{2}(x)\right|^2\,dx\le\frac{m^2}{4c_1}\left(\mh[\rhotne]-\mh[\rhotnne]\right).
\]
By summing up over $n$ from $0$ to $N-1$, taking into account~\Cref{rem:control_below_entropy} and that second order moments are uniformly bounded (see~\Cref{prop:en-ineq-mom-bound}), we get
\begin{equation}\label{eq:bound-nablaveps}
	\int_0^T\int_{\Rd}\left|\nabla| V_\varepsilon*\rhote(t)|^\frac{m}{2}(x)\right|^2\,dx\,dt\le \frac{m^2}{4c_1}\left(\mh[\rho_0]-\mh[\rhotne]\right)\le \frac{m^2}{4c_1}\left(\mh[\rho_0] + C(\rho_0, \, V_1, \, T)\right).
\end{equation}
For $m=1$, the initial entropy is assumed to be bounded. For $m>1$, since $x\log x \le x^m$ for any $x\ge 0$, we always have $\mh[\rho_0]\le \|\rho_0\|_{L^m}^m$. In both cases, the initial entropy is bounded, and this establishes the desired $L_{t,\,x}^2$ bound for $\nabla(v_\tau^\varepsilon)^\frac{m}{2}$.
\end{proof}

The strong $L^m$ compactness in time and space follows by applying a refined version of the Aubin-Lions Lemma due to Rossi and Savar\'{e} \cite[Theorem 2]{RS}. For the reader's convenience we recall the latter result below before presenting the compactness result for $\{v^{\varepsilon_k}\}_k$.

\begin{prop}\cite[Theorem 2]{RS}\label{prop:aulirs-meas}
Let $X$ be a separable Banach space. Consider
\begin{itemize}
\item a lower semicontinuous functional $\mathscr{F}:X\to[0,+\infty]$ with relatively compact sublevels in $X$;
\item a pseudo-distance $g:X\times X\to[0,+\infty]$, i.e., $g$ is lower semicontinuous and such that $g(\rho,\eta)=0$ for any $\rho,\eta\in X$ with $\mathscr{F}(\rho)<\infty$, $\mathscr{F}(\eta)<\infty$ implies $\rho=\eta$.
\end{itemize}
Let $U$ be a set of measurable functions $u:(0,T)\to X$, with a fixed $T>0$. Assume further that
\begin{equation}\label{hprossav}
\sup_{u\in U}\int_{0}^T\mathscr{F}(u(t))\,dt<\infty\quad \text{and}\quad \lim_{h\downarrow0}\sup_{u\in U}\int_{0}^{T-h}g(u(t+h),u(t))\,dt=0\,.
\end{equation}
Then $U$ contains an infinite sequence $(u_n)_{n\in\mathbb{N}}$ that converges in measure, with respect to $t\in(0,T)$, to a measurable $\tilde{u}:(0,T)\to X$, i.e.
\[
\lim_{n\to\infty}|\{t\in(0,T):\|u_n(t)-u(t)\|_X\ge\sigma\}|=0, \quad \forall \sigma>0.
\]
\end{prop}

The two conditions in \eqref{hprossav} are called \textit{tightness} and \textit{weak integral equicontinuity}, respectively.
\begin{prop}\label{prop:strong-convergence-v}
Fix $m\ge 1$ and consider the family $\{v_\tau^\varepsilon\}_{\varepsilon\in(0,\varepsilon_0), \tau>0}$ in~\Cref{lem:h1-bound}. There is a subsequence $\tau_k\downarrow 0$ such that for any $\varepsilon>0$, we have
\[
v_{\tau_k}^\varepsilon \to v^\varepsilon = V_\varepsilon*\tilde{\rho}^\varepsilon, \quad \text{in }L^m([0,T]\times \R^d).
\]
Moreover, there is a subsequence $\varepsilon_k\downarrow 0$ and a curve $v\in C([0,T];\mptrd)\cap L^m([0,T]\times \R^d)$ such that
\[
v^\varepsilon \to v, \quad \text{in }L^m([0,T]\times\R^d).
\]
\end{prop}
\begin{proof}
	The proof of the result is obtained by applying~\Cref{prop:aulirs-meas} to a subset of the sequence $U:=\{v_\tau^\varepsilon\}_{\varepsilon\in(0,\varepsilon_0), \tau >0}$ for $X:=L^m(\Rd)$ and $g:=d_1$ being the $1$-Wasserstein distance --- extended to $+\infty$ outside of $\mP_1(\Rd)\times\mP_1(\Rd)$. As for the functional, we consider $\mathscr{F}:L^m(\Rd)\to[0,+\infty]$ defined by
	\begin{equation*}
		\mathscr{F}[v]=
		\begin{cases}
			\left\|v^\frac{m}{2}\right\|_{H^1(\Rd)}^2 + \int_{\R^d}|x|v(x)\, dx, & \text{if } v\in\mP_1(\Rd) \mbox{ and } v^{\frac{m}{2}}\in H^1(\Rd);\\
			+\infty, & \text{otherwise}.
		\end{cases}
	\end{equation*}
	Note that elements in the domain of the functional $\mathscr{F}$ belong to $\mP_1(\Rd)$, thus $0=g(\rho,\eta)=d_1(\rho,\eta)$ implies $\rho=\eta$. Let us check that $\mathscr{F}$ is an admissible functional. 
	
	Lower semicontinuity can be easily verified following, e.g., \cite{BE22}. Let $A_c:=\{v\in L^m(\Rd): \mathscr{F}[v]\le c\}$ be a sublevel of $\mathscr{F}$, where $c$ is a positive constant. We consider $B_c:=\{w=v^{\frac{m}{2}}:v\in A_c\}$ and prove that $B_c$ is relatively compact in $L^2(\Rd)$, as the map $w\in L^2(\Rd)\mapsto \iota(w)=w^\frac{2}{m}\in L^m(\Rd)$ is continuous and $A_c=\iota(B_c)$.
	
	The Riesz-Fr\'{e}chet-Kolmogorov theorem provides relatively compactness in $L^2(\Rd)$ of $B_c$. In fact, elements of $B_c$ are bounded in $L^2(\Rd)$ and it holds the uniform continuity estimate
	\begin{equation}\label{eq:l2-cont-est}
		\begin{split}
			\int_{\Rd}|w(x+h)-w(x)|^2dx &\!=\! \int_{\Rd}\left|\int_0^1 \frac{d}{d\tau}w(x+\tau h)\,d\tau\right|^2 dx\! =\! \int_{\Rd}\left|\int_0^1 h \cdot \nabla w(x+\tau h)\,d\tau\right|^2 dx\\
			&\le |h|^2\int_{\Rd}\int_0^1 |\nabla w(x+\tau h)|^2\,d\tau \,dx = |h|^2\|\nabla w\|_{L^2(\Rd)}^2,
		\end{split}
	\end{equation}
	which implies $\|w(\cdot+h)-w(\cdot)\|_{L^2(\Rd)}\to0$ as $h\to0^+$.
	
	Before proceeding to the uniform integrability, we record the following improved estimates afforded to us by the fact that $B_c$ is a bounded subset of $H^1(\R^d)$.
\begin{equation}
	\label{eq:sob_embed}
\sup_{w\in B_c}\|w\|_{L^q(\R^d)} \le c, \quad q \in \left\{
\begin{array}{cc}
\{+\infty\} 	&d=1 	\\

[2,+\infty) 	&d=2 	\\

[2, \frac{2d}{d-2}]	&d>2
\end{array}
\right..
\end{equation}
In the case $d=1$, for any $m\ge 1$, we set $\delta = 1$ in the following estimate
\begin{align}
	\label{eq:unifint1}
	\|w\|_{L^2(\Rd\setminus B_R)}^2&=\int_{|x|\ge R}|v(x)|^m \,dx \le \frac{1}{R^{\delta}}\int_{\Rd}|x|^{\delta}|v(x)|^m\,dx 	\\
	&\le \frac{\|v\|_{L^\infty}^{m-1}}{R}\int_{\R^d}|x|v(x)\, dx\le \frac{\|v\|_{L^\infty}^{m-1}}{R}\mathscr{F}[v]\le \frac{\|v\|_{L^\infty}^{m-1}}{R}c. \notag
\end{align}
Hence, uniform integrability is proven in the case $d=1$ and $m\ge 1$. In fact, for any $d\ge 2$ and $m=1$, we can simply take $\delta=1$ again in~\eqref{eq:unifint1} to establish uniform integrability in this case. For general $d\ge 2$ and $m>1$, we further develop~\eqref{eq:unifint1} by H\"older's inequality to obtain, for a particular choice of $\delta\in(0,1)$ which will be made clear,
\begin{equation}
	\label{eq:unifint2}
\|w\|_{L^2(\R^d\setminus B_R)}^2	\le \frac{1}{R^{\delta}}\left(\int_{\R^d}|x|v(x)\,dx\right)^\delta \left(\int_{\Rd}|v(x)|^{\frac{m-\delta}{1-\delta}}\,dx\right)^{1-\delta}.
\end{equation}
The parameter $\delta\in(0,1)$ can be chosen to take advantage of the extra integrability from~\eqref{eq:sob_embed}. For example, we can take
\[
\delta = \frac{2}{d(m-1)+2} \in (0,1),
\]
which is permissible in light of the Sobolev embedding~\eqref{eq:sob_embed} $\int_{\R^d}|v(x)|^\frac{m-\delta}{1-\delta} \, dx \le c$ recalling $v = w^\frac{2}{m}$.
Thus, \eqref{eq:unifint2} yields the uniform integrability of $w$ in $L^2$.
	
	We now check tightness and weak integral equicontinuity, i.e. conditions \eqref{hprossav}. Let us set $U:=\{v_\tau^\varepsilon\}_{0<\varepsilon\le\varepsilon_0, 0 < \tau}$, being $v_\tau^\varepsilon:[0,T]\to L^m(\Rd)$ the sequence defined above by $v_\tau^\varepsilon=V_\varepsilon*\rhote$, which satisfies~\Cref{lem:h1-bound}. For any $0<\varepsilon\le\varepsilon_0$ and $\tau>0$, it holds
	\begin{align*}
		\int_0^T\mathscr{F}[v_\tau^\varepsilon(t)]\,dt&=\int_0^T\left\|\left(v^\varepsilon_\tau\right)^{\frac{m}{2}}\right\|_{H^1(\Rd)}^2\, dt  + \int_0^T\int_{\R^d}|x|v^\varepsilon_\tau(x)\, dx\,dt\\
		&\le C(\rho_0,V_1,T)+ \varepsilon_0T\int_\Rd V_1(z)|z|\,dz<+\infty,
	\end{align*}
	where we used
	\begin{align*}
		\int_\Rd|x|v^\varepsilon_\tau(x)\,dx&=\iint_\Rdd |x|V_\varepsilon(x-y)\,d\rhote(y)\,dx\\
		&\le \iint_\Rdd V_\varepsilon(x-y)|x-y|\,d\rhote(y)\,dx+ \iint_\Rdd V_\varepsilon(x-y)|y|\,d\rhote(y)\,dx\\
		&=\varepsilon\int_\Rd V_1(z)|z|\,dz+\int_\Rd V_1(z)\,dz\int_\Rd|y|\,d\rhote(y)\\
		&\le\varepsilon_0\int_\Rd V_1(z)|z|\,dz+\sqrt{m_2(\rhote)}\int_\Rd V_1(z)\,dz<+\infty.
	\end{align*}
	due to~\Cref{rem:uniform-bound-initial-energy} and~\Cref{prop:en-ineq-mom-bound}.
	Taking the supremum in $U$ we have tightness. For the weak integral equicontinuity, we fix $\varepsilon, \, h>0$ and consider the $\tau \le h$ and $\tau > h$ cases separately. Starting with $\tau \le h$, we use the almost H\"older continuity of $\rhote$ proven in~\eqref{eq:holder-cont} of~\Cref{prop:limit-rho}. More precisely, it holds
	\begin{align*}
		\int_0^{T-h}\!\!d_1(v_\tau^\varepsilon(t+h),v_\tau^\varepsilon(t))\,dt&\le\!\!\int_0^{T-h}\!\!d_W(v_\tau^\varepsilon(t+h),v_\tau^\varepsilon(t))\,dt\le\!\!\int_0^{T-h}\!\!d_W(\rhote(t+h),\rhote(t))\,dt 	\\
		&\le c\int_0^{T-h} (\sqrt{h} + \sqrt{\tau}) \, dt \le 2c(T-h)\sqrt{h}.
	\end{align*}
	where in the intermediate inequalities we used~\eqref{eq:holder-cont} for some constant $c>0$ (independent of $\varepsilon, \, \tau, \, h$) as well as standard properties of Wasserstein distances, c.f. for example \cite[Section 5.1]{S}. The equicontinuity follows by sending $h\downarrow 0$. In the case $\tau > h$, we use~\eqref{eq:total-square-en-estim} instead to estimate
\begin{align*}
	\quad \int_0^{T-h}\!\!d_1(v_\tau^\varepsilon(t+h),v_\tau^\varepsilon(t))\,dt&\le\!\!\int_0^{T-h}\!\!d_W(\rhote(t+h),\rhote(t))\,dt \le  h\sum_{n=0}^{N-1}d_W(\rhotnne,\rhotne) 	\\
&\le h N^\frac{1}{2}\left(
\sum_{n=0}^{N-1}d_W^2(\rhotnne, \rhotne)
\right)^\frac{1}{2} \le cT^\frac{1}{2} h,
\end{align*}
where the constant $c$ is defined when proving \eqref{eq:total-square-en-estim}.

We are left to prove the relative compactness in $L^m([0,T];L^m(\Rd))$ for all $m\ge 1$. We start with the limit $\tau\downarrow 0$ for fixed $\varepsilon>0$. Remember that the estimates we have proven so far are uniform in $\varepsilon$ and $\tau$ so there is no dependence on $\varepsilon$ as $\tau\downarrow0$. We begin with the $m>1$ case. The first part in the proof of~\Cref{lem:h1-bound} showed that $\|v_\tau^\varepsilon\|_{L^m([0,T]\times \R^d)}$ is uniformly bounded. Thus there exists a subsequence $\tau_k\downarrow 0$ such that $v_{\tau_k}^\varepsilon \rightharpoonup v^\varepsilon$ in $L^m([0,T]\times \R^d)$ for some $v^\varepsilon\in L^m([0,T]\times \R^d)$. By~\Cref{prop:en-ineq-mom-bound}, we know that $\rho_\tau^\varepsilon$ narrowly converges to $\tilde{\rho}^\varepsilon$ along a subsequence uniformly in $[0,T]$. By testing against  smooth functions, we must have agreement between these limits $v^\varepsilon = V_\varepsilon * \tilde{\rho}^\varepsilon$. Moreover, along a further subsequence which we just label $\tau \downarrow0$, we can apply~\Cref{prop:aulirs-meas} giving
	\[
		\lim_{\tau\downarrow0} \left|
		\left\{
		t\in(0,T) \, : \, \|v_\tau^\varepsilon(t) - v^\varepsilon(t)\|_{L^m(\R^d)}\ge \sigma
		\right\}
		\right| = 0, \quad \forall \sigma>0.
	\]
Let us denote the set above by $A_\sigma(\tau)$. For arbitrary $\sigma>0$, we have
\begin{align*}
&\quad \|v_\tau^\varepsilon - v^\varepsilon\|_{L^m([0,T]\times \R^d)}^m= \int_0^T\int_{\R^d} |v_\tau^\varepsilon - v^\varepsilon|^m = \left(\int_{A_\sigma(\tau)} + \int_{[0,T]\setminus A_\sigma(\tau)}\right)\int_{\R^d}|v_\tau^\varepsilon - v^\varepsilon|^m	\\
&\le \sup_{s\in[0,T]}2^{(m-1)}\left(\|v_\tau^\varepsilon(s)\|_{L^m}^m + \|v^\varepsilon(s)\|_{L^m}^m \right)\left|A_\sigma(\tau)\right| + \sigma^m T.
\end{align*}
Similar to~\eqref{eq:lmest}, we can insert
\[
\sup_{\varepsilon,\tau>0,\, t \in[0,T]}\|v_\tau^\varepsilon \|_{L^m}^m \le \frac{c_2}{c_1}\|\rho_0\|_{L^m}^m, \quad\text{and} \quad \sup_{\varepsilon>0, \,t\in[0,T]}\|v^\varepsilon\|_{L^m}^m\le \frac{c_2}{c_1}\|\rho_0\|_{L^m}^m
\]
into the previous estimate to obtain
\[
\|v_\tau^\varepsilon- v^\varepsilon\|_{L^m([0,T]\times \R^d)}^m \le 2^m\frac{c_2}{c_1} \|\rho_0\|_{L^m}^m |A_\sigma(\tau)| + \sigma^mT.
\]
Passing to $\tau\downarrow0$ and using $\lim_{\tau\downarrow0} |A_\sigma(\tau)| = 0$, we arrive at
\[
\limsup_{\tau\downarrow 0}\|v_\tau^\varepsilon- v^\varepsilon\|_{L^m([0,T]\times \R^d)} \le \sigma T^\frac{1}{m}.
\]
Since $\sigma>0$ was arbitrary, this implies the strong $L^m$ convergence from $v_\tau^\varepsilon$ to $v^\varepsilon = V_\varepsilon*\tilde{\rho}^\varepsilon$.

In the case $m=1$, we need to argue differently. We apply~\cite[Proposition 1.10]{RS} which asserts that relatively compactness in $L^1((0,T);X)$ is implied by uniform integrability and relatively compactness in measure as a function with values in $X$ ($X\equiv L^1(\Rd)$ in this proof). Compactness in measure has just been proven as an application of~\Cref{prop:aulirs-meas}. Following~\cite[Remark 1.11]{RS}, uniform integrability is a consequence of the strong integral equicontinuity
	\[
	\lim_{h\to0}\sup_{\tau}\int_0^{T-h}\left|\|v^\varepsilon_\tau(t+h)\|_{L^1(\Rd)}-\|v^\varepsilon_\tau(t)\|_{L^1(\Rd)}\right|\,dt=0,
	\]
	where we used that $\|v^\varepsilon_\tau(t)\|_{L^1(\Rd)}=1$ for any $t\in[0,T]$ and $\varepsilon>0$.

\ul{Strong compactness $\varepsilon\downarrow 0$:} We first claim that the estimate in~\Cref{lem:h1-bound} also holds uniformly for $v^\varepsilon$, namely
\begin{equation}
\label{eq:h1_v^eps}
\sup_{\varepsilon>0} \left\|(v^\varepsilon)^\frac{m}{2}\right\|_{L^2(0,T; \, H^1(\R^d))} \le C(\rho_0,\, V_1, \, T).
\end{equation}
This can be seen by the fact that, up to a further subsequence, $(v_\tau^\varepsilon)^\frac{m}{2}$ converges to $(v^\varepsilon)^\frac{m}{2}$ strongly in $L^2([0,T]\times \R^d)$. Indeed, by standard results in $L^p$ integration theory and the fact that $v_\tau^\varepsilon \to v^\varepsilon$ strongly in $L^m$, there exists $w^\varepsilon\in L^m([0,T]\times\R^d)$ such that, along a subsequence, $|v_\tau^\varepsilon| \le w^\varepsilon$ for almost every $(t,x) \in [0,T]\times \R^d$. Moreover, we have $v_\tau^\varepsilon\to v^\varepsilon$ pointwise almost everywhere in $[0,T]\times \R^d$. Using Lebesgue's dominated convergence theorem, we obtain
\[
\int_0^T\int_{\R^d}\left|(v_\tau^\varepsilon)^\frac{m}{2} - (v^\varepsilon)^\frac{m}{2}\right|^2 \to 0,
\]
since the integrand converges to 0 pointwise almost everywhere and it is majorised, uniformly in $\tau$, by 
\[
|(v_\tau^\varepsilon)^\frac{m}{2} - (v^\varepsilon)^\frac{m}{2}|^2 \le 2((w^\varepsilon)^m + (v^\varepsilon)^m)\in L^1([0,T]\times \R^d).
\]
Owing to the (weak $L^2$) lower semicontinuity of the $H^1$ seminorm, the estimate in~\Cref{lem:h1-bound} passes to the limit (along a subsequence) $\tau\downarrow 0$ and~\eqref{eq:h1_v^eps} is established.

At this point, we can repeat all of the previous argument for $U=\{v^\varepsilon\}_{\varepsilon\in(0,\varepsilon_0)}$. We take the same space $X = L^m(\R^d)$ and $g=d_1$. The same functional $\mathscr{F}$ is still admissible. Tightness and weak integral equicontinuity can be analogously proven.

\Cref{prop:aulirs-meas} applies and we have convergence in measure for $v^\varepsilon$ to some curve $v$ described in the statement of this result. By the same arguments as before, this convergence is strong in $L^m([0,T]; L^m(\R^d))$.

\end{proof}

\section{Convergence of solutions}\label{sec:nonlocal-to-local-limit}
This section addresses the proof of~\Cref{thm:exist_nonlinear_diffusion}. We cover the case $m\ge 2$ in~\Cref{sec:mge2} while the case $1<m<2$ is treated in~\Cref{sec:mle2}. To simplify the presentation, we focus on functionals $\mf = \mh_m$ but we also discuss (see~\Cref{rem:consistency-general-F}) the extension to general energies satisfying~\ref{ass:AGS-F}, \ref{ass:diff-F}, \ref{ass:diff-reg}, and~\ref{ass:F-m} for convergence from~\eqref{eq:nlie-class} to~\eqref{eq:nonlinear-diffusion}.
\subsection{The case $m\ge 2$}
\label{sec:mge2}
Building on the previous discussions from~\Cref{sec:nlie,sec:compact}, we denote $\rhoe$ the weak measure solutions to~\eqref{eq:nlie} constructed from the JKO scheme in~\Cref{prop:en-ineq-mom-bound}. Moreover, we focus on the subsequence such that $v^\varepsilon = V_\varepsilon * \rhoe$ converges to $v\in C([0,T];\, \mptrd)\cap L^m([0,T]\times \Rd)$ in $L^m([0,T]\times \Rd)$ from~\Cref{prop:strong-convergence-v}. Starting from the definition of weak measure solution to \eqref{eq:nlie} we can reformulate the right-hand side as follows:
\begin{equation}\label{eq:weak-form-conv}
\begin{split}
\int_\Rd\varphi(x)d\rho_t^\varepsilon(x)\!-\!\int_\Rd\varphi(x)d\rho_0(x)&=-\frac{m}{m-1}\int_0^t\int_\Rd\nabla\varphi(x)\cdot\nabla V_\varepsilon*(V_\varepsilon*\rho_r^\varepsilon)^{m-1}(x)d\rho_r^\varepsilon(x)dr\\
&=-\frac{m}{m-1}\int_0^t\int_\Rd (V_\varepsilon*\rho_r^\varepsilon\nabla\varphi)(x)\nabla(V_\varepsilon*\rho_r^\varepsilon)^{m-1}(x)\,dx\,dr\\
&=-2\int_0^t\int_\Rd(V_\varepsilon*\rho_r^\varepsilon\nabla\varphi)(V_\varepsilon*\rho_r^\varepsilon)^{\frac{m}{2}-1}\nabla(V_\varepsilon*\rho_r^\varepsilon)^{\frac{m}{2}}\,dx\,dr\\
&=-2\int_0^t\int_\Rd\nabla\varphi(x)(V_\varepsilon*\rho_r^\varepsilon)^{\frac{m}{2}}(x)\nabla(V_\varepsilon*\rho_r^\varepsilon)^{\frac{m}{2}}(x)\,dx\,dr\\
&\quad-2\int_0^t\int_\Rd z_r^\varepsilon(x)(V_\varepsilon*\rho_r^\varepsilon)^{\frac{m}{2}-1}\nabla(V_\varepsilon*\rho_r^\varepsilon)^{\frac{m}{2}}(x)\,dx\,dr,
\end{split}
\end{equation}
being for any $r\in[0,T]$ and $x\in\Rd$, the error term
\begin{equation}\label{eq:excess-term}
z_r^\varepsilon(x):=(V_\varepsilon*\rho_r^\varepsilon\nabla\varphi)(x)-\nabla\varphi(x)(V_\varepsilon*\rho_r^\varepsilon)(x).
\end{equation}
The product $(V_\varepsilon* \rho_r^\varepsilon)^\frac{m}{2}\nabla (V_\varepsilon*\rho_r^\varepsilon)^\frac{m}{2}$ in the last line of~\eqref{eq:weak-form-conv} is a weak-strong convergence pair in $L_{t,x}^2$. Indeed, recall the uniform $L_t^2H_x^1$ bound on $(V_\varepsilon * \rho^\varepsilon)^\frac{m}{2}$ and~\Cref{prop:aulirs-meas} from~\Cref{sec:compact}. Hence, the first integral in the last line of~\eqref{eq:weak-form-conv} passes well in the limit (along a subsequence) $\varepsilon\to 0$. This is precised later in the full proof of~\Cref{thm:exist_nonlinear_diffusion} so we dedicate much of this section to estimates proving that the error vanishes as $\varepsilon\to 0$.
\begin{rem}
If $V_1$ is compactly supported, the argument that the last term in~\eqref{eq:weak-form-conv} vanishes as $\varepsilon\downarrow 0$ can be simplified based on the arguments in~\Cref{sec:mle2}. In the rest of this subsection however, we present a general argument allowing for $V_1$ with unbounded support.
\end{rem}
Notice that the last term in the last equality of~\eqref{eq:weak-form-conv} can be estimated as
\[
\|z^\varepsilon (v^\varepsilon)^{\frac{m}{2}-1}\nabla (v^\varepsilon)^{\frac{m}{2}}\|_{L^1([0,t]\times\Rd)}\le\|z^\varepsilon\|_{L^m([0,t]\times\Rd)}\|(v^\varepsilon)^{\frac{m}{2}-1}\|_{L^q([0,t]\times\Rd)}\|\nabla (v^\varepsilon)^{\frac{m}{2}}\|_{L^2([0,t]\times\Rd)},
\]
for $q=\frac{2m}{m-2}$, so that $\frac{1}{m}+\frac{1}{q}+\frac{1}{2}=1$ and
\[
\|(v^\varepsilon)^{\frac{m}{2}-1}\|_{L^q([0,t]\times\Rd)}=\|v^\varepsilon\|_{L^m([0,t]\times\Rd)}^{m\left(\frac{m-2}{2}\right)}\le c(T,V_1,\rho_0,m).
\]
Notice that the exponent $q$ is only valid for $m\ge 2$ based on the computations above. In order to obtain a solution of \eqref{eq:pme} in the $\varepsilon\to0^+$ limit, we need to prove that $z^\varepsilon\to0$ in $L^m([0,t]\times\Rd)$, for any $t\in[0,T]$. In turn, this will imply the error term in \eqref{eq:weak-form-conv} vanishes as $\varepsilon\to0$, as a consequence of the $L^2$ version of Lebesgue dominated convergence theorem and weak-$L^2$ convergence. More precisely, we note that the product $z^\varepsilon(v^\varepsilon)^{\frac{m}{2}-1}\in L^2([0,t]\times\Rd)$ and it converges to $0$ strongly in $L^2$.

\begin{lem}\label{lem:convergence-z-0} 
There exists a vanishing subsequence $\varepsilon_k$ such that the error term $z^{\varepsilon_k}$ converges to zero in $L^m([0,T]\times\Rd)$ as $\varepsilon_k\to0$. 
\end{lem}
\begin{proof}
First we notice that for any $t\in[0,T]$ and $\varphi\in C^2_c(\Rd)$ it holds
\begin{align*}
    \int_\Rd |z^\varepsilon_t(x)|\,dx &\le \int_{\Rd}\int_{\Rd} V_\varepsilon(x-y) |\nabla \varphi(y) - \nabla \varphi(x)| d\rhoe_t(y)\,dx\\
    &\le\|D^2\varphi\|_\infty\int_{\R^d}\int_{\Rd}V_\varepsilon(x-y)|y-x|\,d\rhoe_t(y)\,dx\\
    &=\varepsilon\|D^2\varphi\|_\infty\int_{\Rd}|z|V_1(z)\,dz,
\end{align*}
by means of the change of variable $z=\frac{x-y}{\varepsilon}$. Therefore, there exists a constant $C(V_1,\varphi)$ such that $\|z^\varepsilon\|_{L^\infty([0,T];L^1(\Rd))}\le\varepsilon C(V_1,\varphi)$, whence, up to passing to a subsequence, $z_t^\varepsilon(x)\to0$ for a.e. $(t,x)\in [0,T]\times\Rd$. We now find a majorant to apply the $L^p$ version of the generalised Lebesgue dominated convergence theorem. 

For almost every $x\in\Rd$ and $t\in[0,T]$, for $i=1,\ldots,d$, the non-negativity of $V_1$ and $\rhoe_t$ gives
$$   \left\vert \int_{\Rd} V_\varepsilon(x-y) \partial_{x_i} \varphi(y) d\rhoe_t(y) \right\vert 
\leq \int_{\Rd} V_\varepsilon(x-y) \vert \partial_{x_i} \varphi(y)  \vert d\rhoe_t(y) \leq 
\Vert \partial_{x_i} \varphi \Vert_{\infty} ~v^\varepsilon_t(x),$$
whence
$$
|z_t^\varepsilon(x)|\le2\|\nabla\varphi\|_\infty|v_t^\varepsilon(x)|.
$$
Since $v^\varepsilon\in L^m([0,t]\times\Rd)$ and it converges strongly in $L^m$, c.f.~Proposition \ref{prop:strong-convergence-v}, we are able to conclude the result, as aforementioned.
\end{proof}

\begin{lem}\label{lem:limit-dist}
For any $t\in[0,T]$ and any $\varphi\in C_c^1(\Rd)$ it holds
$$
\lim_{\varepsilon\to0^+}\int_\Rd \varphi(x)v_t^\varepsilon(x)\,dx=\int_{\Rd}\varphi(x)\,d\tilde{\rho}(t).
$$
\end{lem}
\begin{proof}
For any $t\in[0,T]$ and any $\varphi\in C_c^1(\Rd)$, by using the definition of $v_t^\varepsilon$ we obtain:
\begin{align*}
\left|\int_\Rd\varphi(x)v_t^\varepsilon(x)\,dx-\int_\Rd\varphi(x)\,d\rhoe_t(x)\right|&=\left|\int_\Rd\varphi(x)(V_\varepsilon*\rhoe_t)(x)\,dx-\int_\Rd\varphi(x)\,d\rhoe_t(x)\right|\\
&=\left|\int_\Rd(\varphi*V_\varepsilon)(x)\,d\rhoe_t(x)-\int_\Rd\varphi(x)\,d\rhoe_t(x)\right|\\
&=\left|\int_\Rd[(\varphi*V_\varepsilon)(x)-\varphi(x)]\,d\rhoe_t(x)\right|\\
&\le\int_\Rd\int_\Rd|\varphi(x-y)-\varphi(x)|V_\varepsilon(y)\,dy\,d\rhoe_t(x)\\
&\le\|\nabla\varphi\|_\infty\int_\Rd |y|V_\varepsilon(y)\,dy\\
&=\varepsilon\|\nabla\varphi\|_\infty\int_\Rd |x|V_1(x)\,dx,
\end{align*}
which converges to $0$ as $\varepsilon\to0^+$ since $\int_\Rd|x|V_1(x)\,dx<+\infty$. In the second last estimate, we used the mean-value inequality $|\varphi(x-y) - \varphi(x)| \le \|\nabla \varphi\|_{\infty} |y|$.
\end{proof}
We now have all the information to prove~\Cref{thm:exist_nonlinear_diffusion} in the case $\mf = \mh_m$ for $m\ge 2$.

\begin{proof}[Proof of~\Cref{thm:exist_nonlinear_diffusion} for $\mf = \mh_m$ and $m\ge 2$]
Since $\rhoe$ is a weak solution to \eqref{eq:nlie}, for any $\varphi\in C^1_c(\Rd)$ and $t\in[0,T]$ it satisfies
\begin{align*}
    \int_\Rd\varphi(x)\,d\rho_t^\varepsilon(x)-\int_\Rd\varphi(x)\,d\rho_0(x)&=-2\int_0^t\int_\Rd\nabla\varphi(x)(V_\varepsilon*\rho_r^\varepsilon)^{\frac{m}{2}}(x)\nabla(V_\varepsilon*\rho_r^\varepsilon)^{\frac{m}{2}}(x)\,dx\,dr\\
&\quad-2\int_0^t\int_\Rd z_r^\varepsilon(x)(V_\varepsilon*\rho_r^\varepsilon)^{\frac{m}{2}-1}\nabla(V_\varepsilon*\rho_r^\varepsilon)^{\frac{m}{2}}(x)\,dx\,dr,
\end{align*}
as explained in~\eqref{eq:weak-form-conv}. \Cref{prop:limit-rho}, \Cref{lem:h1-bound}, \Cref{lem:limit-dist}, and~\Cref{prop:strong-convergence-v} infer existence of a subsequence of $\rhoe(t)$ narrowly converging to $\tilde{\rho}\in L^m([0,T];L^m(\Rd))$, and, in particular, $\{v^\varepsilon\}_\varepsilon$ admits a subsequence such that
\begin{align*}
    &v^{\varepsilon_k}\to\tilde{\rho} \qquad\quad \mbox{ in } L^m([0,T];L^m(\Rd));\\
    \nabla& (v^{\varepsilon_k})^\frac{m}{2}\rightharpoonup w \quad \, \mbox{ in } L^2([0,T];L^2(\Rd)).
\end{align*}
By a standard argument one can show that $(v^{\varepsilon_k})^\frac{m}{2}\to(\tilde{\rho})^\frac{m}{2}$ in $L^2([0,T];L^2(\Rd))$, whence $w\equiv\nabla(\tilde{\rho})^{\frac{m}{2}}$.
Before letting $\varepsilon\to0^+$ and obtaining the result we need to further regularise the test function, $\varphi$, since in~\Cref{lem:convergence-z-0} we make use of test functions in $C^2_c(\Rd)$. In this regard, we consider a standard mollifier $\eta\in C_c^\infty(\Rd)$ and the corresponding sequence $\varphi^\sigma:=\eta^\sigma*\varphi\in C_c^\infty(\Rd)$, being $\eta^\sigma(x)=\sigma^{-d}\eta(x/\sigma^d)$ for any $x\in\Rd$ and $\sigma>0$. As a consequence of the observations above and~\Cref{lem:convergence-z-0}, by letting $\varepsilon\to0^+$ we obtain, for any $\sigma>0$ and $t\in[0,T]$,
\begin{align*}
     \int_\Rd\varphi^\sigma(x)\tilde\rho(t,x)\,dx&= \int_\Rd\varphi^\sigma(x)\rho_0(x)\,dx-2\int_0^t\int_\Rd[\tilde\rho(s,x)]^{\frac{m}{2}} \nabla \varphi^\sigma(x)\cdot \nabla [\tilde\rho(s,x)]^{\frac{m}{2}}\,dx\,ds\\
     &=\int_\Rd\varphi^\sigma(x)\rho_0(x)\,dx-\frac{m}{m-1}\int_0^t\int_\Rd\tilde\rho(s,x) \nabla \varphi^\sigma(x)\cdot \nabla [\tilde\rho(s,x)]^{m-1}\,dx\,ds,
    \end{align*}
where in the last equality we are using $m\ge2$, hence the chain rule holds true, cf.~Remark~\ref{rem:chain_rule_weak_form}. More precisely, we re-write $\rho^{m-1}=G\circ u$, for $u=\rho^{\frac{m}{2}}$ and $G(x)=x^{\frac{2(m-1)}{m}}$ since $\frac{2(m-1)}{m}\geq 1$. As pointed out in Remark~\ref{rem:chain_rule_weak_form}, the usual definition of weak solution holds by identifying $\nabla\rho^m=2\rho ^{\frac{m}{2}}\nabla \rho^{\frac{m}{2}}$ (in the weak sense) --- write $\rho^m=G\circ u$, for $u=\rho^{\frac{m}{2}}$ and $G(x)=x^2$. Since $\varphi^\sigma$ converges uniformly to $\varphi$ on compact sets, we can let $\sigma\to0$ and obtain that $\tilde\rho$ is a weak solution to~\eqref{eq:pme} in the sense of~\Cref{def:sol-pme2}. Uniqueness of weak solutions of \eqref{eq:pme} is a known result, c.f. e.g.~\cite{DalKen84,Vaz07}. Hence, we obtain convergence of the whole sequence $\rhoe$ narrowly converges to $\tilde{\rho}$, and $\tilde{\rho}^{-\frac{1}{2}}|\nabla \tilde{\rho}^m|\in L^1([0,T];L^2(\Rd))$, by comparison with the theory in \cite{AGS}.
\end{proof}

\begin{rem}[\Cref{thm:exist_nonlinear_diffusion} for general functionals]
\label{rem:consistency-general-F}
For a general integrand $F$ satisfying~\ref{ass:AGS-F}, \ref{ass:diff-F}, and~\ref{ass:diff-reg}, the RHS of~\eqref{eq:weak-form-conv} becomes
\begin{align*}
-\int_0^t\int_\Rd\nabla\varphi(x)\nabla F'(V_\varepsilon*\rho_r^\varepsilon)(x)(V_\varepsilon*\rho_r^\varepsilon)(x)\,dx\,dr
-\int_0^t\int_\Rd z_r^\varepsilon(x)\nabla F'(V_\varepsilon*\rho_r^\varepsilon)\,dx\,dr,
\end{align*}
being $z^\varepsilon$ as in \eqref{eq:excess-term}. Supposing $F$ also satisfies~\ref{ass:F-m} for some $m\ge 2$, we have the estimate $|F''(x)|\le c x^{m-2}$, hence $F\in C^{2}([0,\infty))$ (origin included). Notice that the power laws for $m\ge 2$ satisfy all of~\ref{ass:AGS-F}, \ref{ass:diff-F}, \ref{ass:diff-reg}, and~\ref{ass:F-m}. The error term can be estimated as
\begin{align*}
    \left|\int_0^t\int_\Rd z_r^\varepsilon(x)\nabla F'(V_\varepsilon*\rho_r^\varepsilon)\,dx\,dr\right|&\lesssim\int_0^t\int_\Rd|z_r^\varepsilon(x)||v_r^\varepsilon(x)|^{\frac{m}{2}-1}|\nabla (v_r^\varepsilon(x))^{\frac{m}{2}}|\,dx\,dr\\
    &\lesssim\|z^\varepsilon\|_{L^m([0,t]\times\Rd)}\|(v^\varepsilon)^{\frac{m}{2}-1}\|_{L^q([0,t]\times\Rd)}\|\nabla (v^\varepsilon)^{\frac{m}{2}}\|_{L^2([0,t]\times\Rd)},
\end{align*}
for $q=\frac{2m}{m-2}$, so that $\frac{1}{m}+\frac{1}{q}+\frac{1}{2}=1$; thus it vanishes as $\varepsilon\to0$ similar to~\Cref{lem:convergence-z-0}. As for the first term, note that it can be rewritten as
\begin{align*}
 \int_0^t\int_\Rd\nabla\varphi(x)\nabla F'(v_r^\varepsilon)(x)v_r^\varepsilon(x)\,dx\,dr=\frac{2}{m}\int_0^t\int_\Rd F''(v_r^\varepsilon(x))(v_r^\varepsilon(x))^{2-\frac{m}{2}}\nabla (v_r^\varepsilon)^\frac{m}{2}\nabla\varphi(x)\,dx\,dr,
\end{align*}
where $F''(x)x^{2-\frac{m}{2}}$ is extended by zero when $x=0$ owing to~\ref{ass:F-m}, and we applied the chain rule twice on the set $\{v_r^\varepsilon>0\}$
\[
v_r^\varepsilon\nabla v_r^\varepsilon=\frac{1}{2}\nabla[ (v_r^\varepsilon)^{\frac{m}{2}}]^{\frac{4}{m}}=\frac{2}{m}(v_r^\varepsilon(x))^{2-\frac{m}{2}}\nabla (v_r^\varepsilon)^\frac{m}{2}.
\]
When multiplied with $F''(v_r^\varepsilon)$, the integrand on the right-hand side makes sense in $L^1$ owing to~\ref{ass:F-m} since $F''(x) \le c x^{m-2}$. Then, we are left to show $g(v^\varepsilon):=F''(v^\varepsilon)(v^\varepsilon)^{2-\frac{m}{2}}$ strongly converges in $L^2([0,T]\times\Rd)$. This is indeed achieved by bounding $|F''(v^\varepsilon)(v^\varepsilon)^{2-\frac{m}{2}}|\le c (v^\varepsilon)^\frac{m}{2}$ and applying the generalised version of the Lebesgue dominated convergence theorem. Therefore, in the $\varepsilon\to 0^+$ limit we obtain (up to pass to a subsequence)
\[
\int_\Rd\varphi(x)\tilde\rho(t,x)\,dx= \int_\Rd\varphi(x)\rho_0(x)\,dx-\frac{2}{m}\int_0^t\int_\Rd F''(\tilde\rho(s,x))(\tilde\rho(s,x))^{2-\frac{m}{2}}\nabla (\tilde\rho(s,x))^\frac{m}{2}\nabla\varphi(x)\,dx\,dr.
\]
Defining $G$ such that $G'(x^{\frac{m}{2}})=F''(x)x^{2-\frac{m}{2}}$, or $G'(x)=F''(x^\frac{2}{m})x^{\frac{2}{m}(2-\frac{m}{2})}$, and $G(0)=0$, we can apply the chain rule to $P(x)=G(x^{\frac{m}{2}})$ to obtain that $\nabla P(\rho)\in L^1([0,T]\times\Rd)$, thus
\[
\int_\Rd\varphi(x)\tilde\rho(t,x)\,dx= \int_\Rd\varphi(x)\rho_0(x)\,dx-\int_0^t\int_\Rd \nabla P(\tilde\rho(s,x))\nabla\varphi(x)\,dx\,dr.
\]
Note that the chain rule and the construction of the pressure $P$ holds for all $m>1$. Uniqueness of distributional solutions of~\eqref{eq:nonlinear-diffusion} is proven in \cite{Brezis_Crandall_79} for bounded solutions, which is actually the case for $L^1$ solutions --- see~\cite{Veron_79,JAC_DF_TOSCANI_ARMA_06} for further details on the so-called $L^1$-$L^\infty$ regularising effect. In particular, from \cite[Theorem 11.2.5]{AGS} we infer that our solution is a $2$-Wasserstein gradient flow satisfying 
\[
\int_0^T\int_\Rd\frac{|\nabla P(\rho)|^2}{\rho}\,dx\,dt<\infty.
\]
Furthermore, uniqueness of solutions implies convergence of the whole sequence $\rho^\varepsilon$, as for \eqref{eq:pme}.
\end{rem}
\begin{rem}
Our result can be also interpreted in the context of \textit{generalised gradient flows} or \textit{gradient structures}, following the dynamical interpretation of the Wasserstein distance, c.f.~\cite{BB2000,DNS09}. More precisely, we know 
\begin{align*}
d_W^2(\rho_0,\rho_1)=\inf\left\{\int_0^1 \mathcal{A}(\rho_t,j_t)\,dt : (\rho,j) \mbox{ solves } \begin{cases}\partial_t\rho+\nabla\cdot j=0,\\
\rho(0)=\rho_0, \quad \rho(1)=\rho_1
\end{cases}\right\},
\end{align*}
where, for any $\lambda\in\mathcal{M}(\Rd;\Rd)$ such that $\rho, j\ll |\lambda|$,
\begin{align*}
\mathcal{A}(\rho,j)=\int_\Rd\alpha\left(\frac{dj}{d|\lambda|},\frac{d\rho}{d|\lambda|}\right)d|\lambda|, \quad \mbox{and} \quad
\alpha(j,r):=\begin{cases}
\frac{(j)^2}{r} \qquad &\text{if}\ r>0,\\
0 \qquad &\text{if}\ j\leq 0\ \text{and}\ r=0,\\
\infty \qquad &\text{if}\ j> 0\ \text{and}\ r=0.
\end{cases}
\end{align*}
Upon using a careful regularisation and cut-off argument one can prove the following chain rule for any absolutely continuous curve with respect to the Wasserstein distance
\[
\mf^\varepsilon[\rho(t)]-\mf^\varepsilon[\rho_0]=-\int_0^t\int_\Rd\nabla\frac{\delta\mf^\varepsilon}{\delta\rho}(x)\, \cdot d j_t(x)\,dt,
\]
hence re-intepret weak (measure) solutions of \eqref{eq:nlie-class} as the zero level set of the De Giorgi functional
\[
\mf^\varepsilon[\rho(t)]-\mf^\varepsilon[\rho_0]+\frac{1}{2}\int_0^t\int_\Rd\left|\nabla\frac{\delta\mf^\varepsilon}{\delta\rho}(x)\right|^2\,d \rho_t(x)\,dt+\frac{1}{2}\int_0^t\mathcal{A}\left(\rho,\,-\rho\nabla\frac{\delta\mf^\varepsilon}{\delta\rho}\right)dt=0.
\]
\end{rem}

\subsection{The case $1<m<2$}
\label{sec:mle2}

The key idea here is to estimate the error $z^\varepsilon$ from~\eqref{eq:excess-term} differently by exploiting the compact support of $V_1$. We take $R>0$ such that $\supp V_1 \subset B_R$. In the case $m\ge 2$, negative powers of $v^\varepsilon = V_\varepsilon*\rho^\varepsilon$ never appeared in~\eqref{eq:weak-form-conv} but these computations can be recycled by cautiously avoiding the 0 level set of $v^\varepsilon$. We define
\[
A_r^\varepsilon := \{x\in\Rd \, | \, v_r^\varepsilon(x) = V_\varepsilon * \rho_r^\varepsilon(x) >0\}
\]
and alter the computations in~\eqref{eq:weak-form-conv} carefully
\begin{align}
    \label{eq:weak-form-conv-concise_1}
    \begin{split}
\int_{\R^d}\!\!\varphi d\rho_t^\varepsilon\!-\!\!\int_{\Rd}\!\!\varphi d\rho_0&=-\frac{m}{m-1}\int_0^t\int_\Rd\nabla\varphi(x)\cdot\nabla V_\varepsilon*(V_\varepsilon*\rho_r^\varepsilon)^{m-1}(x)d\rho_r^\varepsilon(x)dr\\
&=\frac{m}{m-1}\int_0^t\int_{\Rd}\!\!\!(V_\varepsilon*\rho_r^\varepsilon)^{m-1}(x)[\nabla V_\varepsilon*(\rho_r^\varepsilon \nabla \varphi)](x)dxdr\\
&= \frac{m}{m-1}\int_0^t\int_{\Rd}\!\!\!(V_\varepsilon*\rho_r^\varepsilon)^{m-1}(x)\nabla [V_\varepsilon*(\rho_r^\varepsilon \nabla \varphi)](x)dxdr.
    \end{split}
\end{align}
In the second line, we swapped the convolution against $\nabla V_\varepsilon$ and picked up a minus sign because it is an odd function (remember from~\ref{ass:v1} that $V_1$ is even). At this point, we would like to perform integration by parts and apply the gradient onto $(V_\varepsilon*\rho_r^\varepsilon)^{m-1}$. In contrast to~\eqref{eq:weak-form-conv}, when $1<m<2$ we need to avoid the zero set of $V_\varepsilon*\rho_r^\varepsilon$; as smooth as this convolution may be, the function $(V_\varepsilon*\rho_r^\varepsilon)^{m-1}$ is not differentiable on $\Rd \setminus A_r^\varepsilon$. However, due to~\Cref{lem:ibp} (see below), we can justify the integration by parts and develop~\eqref{eq:weak-form-conv-concise_1} to get
\begin{align}
\label{eq:weak-form-conv-concise_2}
\begin{split}
    \int_{\Rd}\!\!\varphi d\rho_t^\varepsilon\!-\!\!\int_{\Rd}\!\!\varphi d\rho_0&=\frac{m}{m-1}\int_0^t\int_{\Rd}\!\!\!(V_\varepsilon*\rho_r^\varepsilon)^{m-1}(x)\nabla [V_\varepsilon*(\rho_r^\varepsilon \nabla \varphi)](x)dxdr \\
    &= -\frac{m}{m-1}\int_0^t \int_{A_r^\varepsilon}\nabla (V_\varepsilon*\rho_r^\varepsilon)^{m-1}(x) V_\varepsilon*(\rho_r^\varepsilon\nabla \varphi)(x)dxdr \quad (\text{\Cref{lem:ibp}})   \\
    &= -2\int_0^t\int_{A_r^\varepsilon}(V_\varepsilon*\rho_r^\varepsilon\nabla\varphi)(V_\varepsilon*\rho_r^\varepsilon)^{\frac{m}{2}-1}\nabla(V_\varepsilon*\rho_r^\varepsilon)^{\frac{m}{2}}\,dx\,dr\\
&=-2\int_0^t\int_{\Rd}\nabla\varphi(x)(V_\varepsilon*\rho_r^\varepsilon)^{\frac{m}{2}}(x)\nabla(V_\varepsilon*\rho_r^\varepsilon)^{\frac{m}{2}}(x)\,dx\,dr\\
&\quad-2\int_0^t\int_{A_r^\varepsilon}z_r^\varepsilon(x)(V_\varepsilon*\rho_r^\varepsilon)^{\frac{m}{2}-1}\nabla(V_\varepsilon*\rho_r^\varepsilon)^{\frac{m}{2}}(x)\,dx\,dr.
\end{split}
\end{align}
The last line is nearly identical to the end result of~\eqref{eq:weak-form-conv}. Here, we integrate over $A_r^\varepsilon = \{ V_\varepsilon*\rho_r^\varepsilon>0\}$ which is justified by~\Cref{lem:ibp}. As was the case in~\Cref{sec:mge2}, we need to show that the error term in the last line vanishes as $\varepsilon\downarrow 0$.
\begin{proof}[Proof of~\Cref{thm:exist_nonlinear_diffusion} for $1<m<2$] 
Convergence in the first term on the right-hand side of \eqref{eq:weak-form-conv-concise_2} can be treated as for the case $m\ge2$, due to Lemma~\ref{lem:h1-bound} and Propositon~\ref{prop:strong-convergence-v}. Hence we focus on the error term. For simplicity, let us assume $\varphi\in C_c^2(\R^d)$ since it can be approximated in such a way as described in~\Cref{sec:mge2}. We estimate the error by first expressing it as
\begin{align*}
    z^\varepsilon(x) &= V_\varepsilon*(\rho^\varepsilon\nabla \varphi)(x) - (V_\varepsilon*\rho^\varepsilon)(x)\nabla\varphi(x) \\
    &= \int_{\Rd}V_\varepsilon(x-y)(\nabla \varphi(y) - \nabla \varphi(x))d\rho^\varepsilon(y).
\end{align*}
Next, we apply the Mean Value theorem to the difference $|\nabla \varphi(y) - \nabla \varphi(x)| \le \|D^2\varphi\|_{L^\infty}|x-y|$ and obtain
\begin{align*}
    |z^\varepsilon| \le \|D^2\varphi\|_{L^\infty}\int_{\R^d} |x-y|V_\varepsilon(x-y)d\rho^\varepsilon(y).
\end{align*}
Now, we exploit the compact support of the generator $V_1$. Since $V_1$ is supported within $B_R$, then $V_\varepsilon$ is supported within $B_{R\varepsilon}$ which leads to
\begin{align}
\label{eq:zeps_new_estimate}
\begin{split}
    &\quad |z^\varepsilon(x)| \le \|D^2\varphi\|_{L^\infty} \int_{\{y \in \R^d \, | \, |x-y|\le R\varepsilon\}}|x-y| V_\varepsilon(x-y)d\rho^\varepsilon(y) \\
    &\le R\varepsilon\|D^2\varphi\|_{L^\infty} \int_{\Rd} V_\varepsilon(x-y)d\rho^\varepsilon(y) = R\varepsilon\|D^2\varphi\|_{L^\infty}(V_\varepsilon*\rho^\varepsilon)(x), \quad \forall x\in\Rd.
\end{split}
\end{align}
The last integral in~\eqref{eq:weak-form-conv-concise_2} can be estimated with~\eqref{eq:zeps_new_estimate} as follows
\begin{align*}
    &\quad \left|\int_0^t\int_{A_r^\varepsilon}
    z_r^\varepsilon(x)(V_\varepsilon*\rho_r^\varepsilon)^{\frac{m}{2}-1}\nabla(V_\varepsilon*\rho_r^\varepsilon)^{\frac{m}{2}}(x)\,dx\,dr
    \right|     \\
    &\le R\varepsilon \|D^2\varphi\|_{L^\infty}\int_0^t \int_{A_r^\varepsilon}(V_\varepsilon*\rho_r^\varepsilon)^{\frac{m}{2}}\left|\nabla(V_\varepsilon*\rho_r^\varepsilon)^{\frac{m}{2}}(x)\right|\,dx\,dr \\
    &\le R\varepsilon \|D^2\varphi\|_{L^\infty}\|(v^\varepsilon)^\frac{m}{2}\|_{L^2([0,T]\times\Rd)}\|\nabla(v^\varepsilon)^\frac{m}{2}\|_{L^2([0,T]\times\Rd)}.
\end{align*}
We have suggestively recalled the notation $v^\varepsilon = V_\varepsilon *\rho^\varepsilon$ precisely with~\Cref{lem:h1-bound} in mind; $(v^\varepsilon)^\frac{m}{2}$ is uniformly bounded in $L^2([0,T]; H^1(\R^d))$. Hence, the last integral is uniformly bounded in $\varepsilon$. Moreover, the prefactor of vanishing $\varepsilon$ implies that the last term of~\eqref{eq:weak-form-conv-concise_2} converges to zero in the limit, thus recovering
\[
\int_{\Rd}\varphi(x) d\rho_t(x)=\int_{\Rd}\varphi(x) d\rho_0(x)-2\int_0^t\int_{\Rd}\nabla\varphi(x)(\rho_r)^{\frac{m}{2}}(x)\nabla(\rho_r)^{\frac{m}{2}}(x)\,dx\,dr.
\]
By means of the chain rule for Sobolev spaces, one can prove $\nabla\rho^m=2\rho ^{\frac{m}{2}}\nabla \rho^{\frac{m}{2}}$ (in the weak sense). More precisely, we can see $\rho^m=G\circ u$, for $u=\rho^{\frac{m}{2}}$ and $G(x)=x^2$. The work by Dahlberg and Kenig~\cite{DalKen84} establishes uniqueness of very weak solutions for \eqref{eq:pme}, for $m>1$. As a byproduct, we also infer that our solution is a gradient flow in the sense of \cite[Theorem 11.2.5]{AGS}, meaning 
\[
\int_0^T\int_\Rd\frac{|\nabla\rho^m|^2}{\rho}\,dx\,dt<\infty.
\]

As for general energies induced by $F$ satisfying~\ref{ass:AGS-F}, \ref{ass:diff-F}, \ref{ass:diff-reg}, and~\ref{ass:F-m}, a combination of the same estimates from~\Cref{rem:consistency-general-F} (disregarding the H\"older estimate) and this new technique for treating $z^\varepsilon$ yield the full result.
\end{proof}

We now justify the integration by parts step going from~\eqref{eq:weak-form-conv-concise_1} to~\eqref{eq:weak-form-conv-concise_2}. Recall the notation $v_r^\varepsilon = V_\varepsilon*\rho_r^\varepsilon$.
\begin{lem}
    \label{lem:ibp}
    For fixed $\varepsilon,\, t>0$, and $\varphi \in C_c^1(\Rd)$, there holds
    \begin{equation}
    \label{eq:IBP}
    \int_0^t\int_{\Rd}(v_r^\varepsilon)^{m-1}(x) \nabla V_\varepsilon*(\rho_r^\varepsilon \nabla \varphi)(x) dx dr = -\int_0^t \int_{A_r^\varepsilon} \nabla (v_r^\varepsilon)^{m-1}(x)V_\varepsilon* (\rho_r^\varepsilon \nabla \varphi)(x) dx dr.
    \end{equation}
    In particular, both integrals converge absolutely.
\end{lem}
\begin{proof}
We begin by proving both integrals in~\eqref{eq:IBP} converge absolutely. For the integral on the left-hand side of~\eqref{eq:IBP}, we estimate
\begin{align}
\label{eq:unifint}
\begin{split}
    &\quad \int_0^t\int_{\Rd}\left|(v_r^\varepsilon)^{m-1}(x) \nabla V_\varepsilon*(\rho_r^\varepsilon\nabla \varphi)(x)\right| dxdr  \\ &\le \int_0^t\int_{\Rd}\left(\int_{\Rd}
    V_\varepsilon(x-y)d\rho(y)
    \right)^{m-1}\left(\int_{\Rd}\left|
    \nabla V_\varepsilon(x-z)\nabla\varphi(z)
    \right|d\rho_r^\varepsilon(z)\right) dx dr  \\
    &\le \int_0^t\int_{\Rd}\|V_\varepsilon\|_{L^\infty}^{m-1}\|\nabla \varphi\|_{L^\infty}\int_{\Rd} |\nabla V_\varepsilon(x-z)|d\rho_r^\varepsilon(z) dx  dr \\
    &\le \int_0^t\|V_\varepsilon\|_{L^\infty}^{m-1}\|\nabla \varphi\|_{L^\infty} \|\nabla V_\varepsilon\|_{L^1}dr \le \|V_\varepsilon\|_{L^\infty}^{m-1}\|\nabla \varphi\|_{L^\infty} \|\nabla V_\varepsilon\|_{L^1} t,
\end{split}
\end{align}
where the last line is obtained by Fubini's theorem. Therefore, the integral on the left-hand side of~\eqref{eq:IBP} is absolutely convergent. Turning to the integral on the right-hand side of~\eqref{eq:IBP}, we first record
\begin{align}
    \label{eq:vepsphi}
    \begin{split}
        |V_\varepsilon * (\rho_r^\varepsilon \nabla \varphi)(x)| &\le \int \left|V_\varepsilon(x-y)\nabla\varphi(y)\right|d\rho_r^\varepsilon(y) \le \|\nabla \varphi\|_{L^\infty}\int V_\varepsilon(x-y)d\rho(y)   \\
        &= \|\nabla \varphi\|_{L^\infty}v_r^\varepsilon(x) = \|\nabla \varphi\|_{L^\infty}v_r^\varepsilon(x) \chi_{A_r^\varepsilon}(x),
    \end{split}
\end{align}
where $\chi_{A_r^\varepsilon}(x)$ is the indicator function on the set $A_r^\varepsilon$. With~\eqref{eq:vepsphi}, we obtain
\begin{align*}
    |\nabla (v_r^\varepsilon)^{m-1}(x)| |V_\varepsilon*(\rho_r^\varepsilon\nabla \varphi)(x)| &\le |\nabla (v_r^\varepsilon)^{m-1}(x)| v_r^\varepsilon(x)\chi_{A_r^\varepsilon}(x)\\
    &= (m-1)|\nabla v_r^\varepsilon| (v_r^\varepsilon)^{m-2} v_r^\varepsilon \chi_{A_r^\varepsilon} = \frac{2(m-1)}{m} |\nabla (v_r^\varepsilon)^\frac{m}{2}|(v_r^\varepsilon)^\frac{m}{2}\chi_{A_r^\varepsilon}.
\end{align*}
Recalling~\Cref{lem:h1-bound}, we conclude by comparison that
\[
|\nabla (v_r^\varepsilon)^{m-1}(x)| |V_\varepsilon*(\rho_r^\varepsilon\nabla \varphi)(x)| \in L^1((0,t)\times \Rd),
\]
which shows that the integral on the right-hand side of~\eqref{eq:IBP} is absolutely convergent.

\ul{Integration by parts:} 
For brevity, we drop the subscript $r\in(0,t)$ and the superscript $\varepsilon>0$ so we consider $v = V *\rho$ in place of $v_r^\varepsilon = V_\varepsilon * \rho_r^\varepsilon$. In order to verify~\eqref{eq:IBP}, we fix $\sigma>0$ and a direction $e\in \mathbb{S}^{d-1}$ and look at the following difference quotient
\begin{align*}
    I_\sigma &:= \int_0^t\int_{\Rd} v^{m-1}(x) \frac{V*(\rho \nabla \varphi)(x+\sigma e) - V*(\rho \nabla \varphi)(x)}{\sigma}dx dr.
\end{align*}
Owing to the uniform bound from~\eqref{eq:vepsphi} and Lebesgue's Dominated Convergence Theorem, we have
\[
\lim_{\sigma\to 0} I_\sigma =e \cdot \int_0^t\int_{\Rd}(v_r^\varepsilon)^{m-1}(x) \nabla V_\varepsilon*(\rho_r^\varepsilon \nabla \varphi)(x) dx dr,
\]
recovering the left-hand side of~\eqref{eq:IBP} (along any arbitrary direction $e\in \mathbb{S}^{d-1}$). On the other hand, by changing variables we also have
\begin{align*}
    I_\sigma &= \int_0^t\int_{\Rd}\frac{v^{m-1}(x-\sigma e) - v^{m-1}(x)}{\sigma}V*(\rho \nabla \varphi)(x)dx dr \\
    &= -\int_0^t\int_{ A_r^\varepsilon}\frac{v^{m-1}(x) - v^{m-1}(x-\sigma e)}{\sigma}V*(\rho\nabla \varphi)(x)dxdr.
\end{align*}
We are allowed to restrict the integration region to $A_r^\varepsilon$ due to~\eqref{eq:vepsphi}; if $x\notin A_r^\varepsilon$, then $|V*(\rho\nabla \varphi)(x)| \le \|\nabla \varphi\|_{L^\infty}v(x) = 0$. In order to prove~\eqref{eq:IBP}, we wish to show
\begin{align}
\label{eq:ibp_power}
\begin{split}
    \lim_{\sigma\to 0}I_\sigma &= \lim_{h\to 0}\left(-\int_0^t\int_{ A_r^\varepsilon}\frac{v^{m-1}(x) - v^{m-1}(x-\sigma e)}{\sigma }V*(\rho\nabla \varphi)(x)dxdr\right) \\
    &= -e\cdot \int_0^t \int_{A_r^\varepsilon} \nabla (v_r^\varepsilon)^{m-1}(x)V_\varepsilon* (\rho_r^\varepsilon \nabla \varphi)(x) dx dr.
\end{split}
\end{align}
The strategy is to apply the extended Dominated Convergence Theorem (\Cref{thm:EDCT}) by exhibiting an appropriate sequence of majorants to the integrand in the first line of~\eqref{eq:ibp_power}.

The first step is to remember~\eqref{eq:vepsphi} and estimate the integrand of $I_\sigma$ as follows
\begin{align}
\label{eq:T_1}
\begin{split}
    \left|
    \frac{v^{m-1}(x) - v^{m-1}(x-\sigma e)}{\sigma}V*(\rho \nabla\varphi)(x)
    \right| &\le \|\nabla \varphi\|_{L^\infty} \left|\frac{v^{m-1}(x) - v^{m-1}(x-\sigma e)}{\sigma}\right|\left|V*\rho(x)\right|     \\
    &= \|\nabla \varphi \|_{L^\infty} \left|\frac{v^{m}(x) - v^{m-1}(x-\sigma e)v(x)}{\sigma }\right|.
\end{split}
\end{align}
In the last line, we have distributed $v = V*\rho$ into the difference quotient. We wish to re-express the term $v^{m-1}(x-\sigma e)v(x)$ as $v^{m}(x) + $ error term. For this, we use the Mean Value Theorem to write
\begin{align*}
    &v(x) = v(x-\sigma e) - \int_0^1 \frac{d}{ds}v(x-s\sigma e)ds= v(x-\sigma e) + \sigma e\cdot \int_0^1 \nabla v(x-s\sigma e)ds.
\end{align*}
Substituting this into $v^m(x) - v^{m-1}(x-\sigma e)v(x)$ gives
\begin{align*}
    v^m(x) - v^{m-1}(x-\sigma e)v(x) = v^m(x) - v^{m-1}(x-\sigma e) \left(
    v(x-\sigma e) + \sigma e\cdot \int_0^1\nabla v(x-s\sigma e)ds
    \right).
\end{align*}
Inserting this into~\eqref{eq:T_1} yields the estimate
\begin{align}
\label{eq:diff_quot_MVT}
\begin{split}
    &\quad \left|
    \frac{v^{m-1}(x) - v^{m-1}(x-h\sigma e)}{\sigma }V*(\rho \nabla\varphi)(x)
    \right|     \\ 
    &\le \|\nabla \varphi\|_{L^\infty}\frac{1}{\sigma }\left|
    v^m(x) - v^{m}(x-\sigma e) - v^{m-1}(x-\sigma e)\sigma e\cdot \int_0^1 \nabla V*(\rho\nabla \varphi)(x - s\sigma e)ds
    \right|.
\end{split}
\end{align}
We use the Mean-Value Theorem again with the initial difference (valid since $m>1$)
\[
v^m(x) - v^m(x-\sigma e) = -\int_0^1 \frac{d}{ds}v^m(x-s\sigma e)ds = m\sigma e\cdot \int_0^1 v^{m-1}(x-s\sigma e)\nabla v(x-s\sigma e) ds
\]
and insert this into~\eqref{eq:diff_quot_MVT} to obtain
\begin{align}
\label{eq:EDCT1}
\begin{split}
    &\quad \left|
    \frac{v^{m-1}(x) - v^{m-1}(x-\sigma e)}{\sigma }V*(\rho \nabla\varphi)(x)
    \right|     \\ 
    &\le \|\nabla \varphi\|_{L^\infty}\frac{1}{\sigma }\left|m\sigma e\cdot \int_0^1 v^{m-1}(x-s\sigma e)\nabla v(x-s\sigma e)ds - v^{m-1}(x-\sigma e)\sigma e\cdot \int_0^1 \nabla V*(\rho\nabla \varphi)(x - s\sigma e)ds
    \right|     \\
    &\le \|\nabla \varphi\|_{L^\infty}\left|
    m\int_0^1v^{m-1}(x-s\sigma e)\nabla v(x-s\sigma e) - v^{m-1}(x-\sigma e)\nabla V*(\rho\nabla \varphi)(x-s\sigma e)ds
    \right|  \\
    &\le \|\nabla \varphi\|_{L^\infty}\int_0^1\left|mv^{m-1}(x-s\sigma e) + \|\nabla \varphi\|_{L^\infty} v^{m-1}(x-\sigma e)\right|\left|
    \nabla v(x-s\sigma e)
    \right|ds.  \\
    &\le C_{m, \varphi}\|V\|_{L^\infty}^{m-1} \int_0^1 |\nabla v(x-s\sigma e)|ds.
\end{split}
\end{align}
In the third line of~\eqref{eq:EDCT1}, we eliminated the common factor of $\sigma$ together with the trivial estimate $|e|=1$. In the fourth line of~\eqref{eq:EDCT1}, we estimated similar to~\eqref{eq:vepsphi} the convolution 
\[
|\nabla V*(\rho \nabla \varphi)| \le \|\nabla \varphi \|_{L^\infty}|\nabla V*\rho| = \|\nabla\varphi\|_{L^\infty}|\nabla v|.
\]
In the final line of~\eqref{eq:EDCT1}, we used the following inequality
\[
v^{m-1}(x) = \left(
\int V(x-y)d\rho(y)
\right)^{m-1} \le \|V\|_{L^\infty}^{m-1} \left(
\int d\rho(y)
\right)^{m-1} = \|V\|_{L^\infty}^{m-1}.
\]
We are now in a position to apply~\Cref{thm:EDCT} with $X = (0,t)\times \Rd$ where $y=(r,x) \in X$ and
\begin{align*}
    f^\sigma(r,x)     &= \frac{v^{m-1}(x) - v^{m-1}(x-\sigma e)}{\sigma }V*(\rho \nabla\varphi)(x)\chi_{A_r^\varepsilon}(x),   \\
    g^\sigma (r,x)     &= C_{m, \varphi}\|V\|_{L^\infty}^{m-1} \int_0^1 |\nabla v(x-s\sigma e)|ds.
\end{align*}
Here, $\chi_{A_r^\varepsilon}(x)$ is the indicator function of the set $A_r^\varepsilon$. Remember that we have suppressed the dependence on $r\in(0,t)$ in $\rho$. The first assumption of~\Cref{thm:EDCT} has been verified by the estimate of~\eqref{eq:EDCT1}. We can verify the second assumption of~\Cref{thm:EDCT} since the pointwise limits of $f^\sigma $ and $g^\sigma $ as $\sigma \to 0$ are
\begin{align*}
    f(r,x) = e\cdot\nabla (v)^{m-1}(x) V*(\rho\nabla \varphi)(x)\chi_{A_r^\varepsilon}(x),    \quad
    g(r,x)  = C_{m,\varphi}\|V\|_{L^\infty}^{m-1}|\nabla v(x)|.
\end{align*}
The pointwise limit for $f^\sigma $ is justified since $A_r^\varepsilon$ is an open set. As for the pointwise limit of $g^\sigma $, the usual Dominated Convergence Theorem suffices.

It remains to check the third assumption of~\Cref{thm:EDCT} which we do with Fubini;
\begin{align*}
    &\quad \int_X g^\sigma (r,x) dy = C_{m,\varphi} \|V\|_{L^\infty}^{m-1}\int_0^t\int_{\Rd} \int_0^1|\nabla v (x-s\sigma e)|ds dx dr\\
    &= C_{m,\varphi}\|V\|_{L^\infty}^{m-1} \int_0^t\int_0^1 \int_{\Rd} |\nabla v(x-s\sigma e)|dx ds dr = C_{m,\varphi} \|V\|_{L^\infty}^{m-1}\int_0^t\int_0^1\int_{\Rd}|\nabla v(x)|dx ds dr   \\
    &= C_{m,\varphi}\|V\|_{L^\infty}^{m-1}\int_0^t\int_{\R^d}|\nabla v(x)|dx dr = \int_X g(r,x)dy.
\end{align*}
Therefore, by~\Cref{thm:EDCT}, we have $\int_X f^\sigma (r,x)dy \to \int_X f(r,x)dy$ as $\sigma \to 0$ which is precisely~\eqref{eq:ibp_power};
\begin{align*}
    \lim_{\sigma \to 0}I_\sigma  &= \lim_{\sigma \to 0}\left(-\int_0^t\int_{ A_r^\varepsilon}\frac{v^{m-1}(x) - v^{m-1}(x-\sigma e)}{\sigma }V*(\rho\nabla \varphi)(x)dxdr\right) \\
    &= -\lim_{\sigma \to 0} \int_X f^\sigma (r,x)dy  = -\lim_{\sigma \to 0}\int_X f(r,x)dy   \\
    &= -e\cdot \int_0^t \int_{ A_r^\varepsilon} \nabla (v_r^\varepsilon)^{m-1}(x)V_\varepsilon* (\rho_r^\varepsilon \nabla \varphi)(x) dx dr.
\end{align*}
\end{proof}

\section{Convexity, uniqueness, and particle approximation}
\label{sec:convexity}
In this section, we sketch the argument adapted from~\cite{blob_weighted_craig} to prove~\Cref{thm:uniqueness} and~\Cref{cor:particle_approx}. Recall that we assume~\ref{ass:F-m} with $m>1$. To simplify the exposition, we set $F'(0) = 0$. In view of the assumptions needed for the kernel $V_1$, see \ref{ass:v1}, it is not reasonable to choose $V_1$ convex, as we require finite second order moment. However, this does not prohibit $\lambda$-convexity of the functional $\mf^\varepsilon$ along geodesics. Indeed, differentiability of $\mf^\varepsilon$, as in~\cite[Proposition 3.10]{Patacchini_blob19}, holds in our case assuming $F$ satisfies \ref{ass:AGS-F}, \ref{ass:diff-F}, \ref{ass:diff-reg}, and it is convex (convexity is ensured by~\ref{ass:F-m}). Furthermore, we need $V_1\in C^2(\Rd)$ and $D^2 V_1\in L^\infty(\Rd)$. In order to give a few explanations in this direction, fix $\rho_1,\rho_2 \in \mptrd$ and consider the geodesic connecting $\rho_1$ to $\rho_2$ defined by
\[
\rho_\alpha := ((1-\alpha)\pi^1 + \alpha\pi^2)_\# \gamma, \quad \alpha\in[0,1],
\]
where $\gamma \in \Gamma(\R^d\times \R^d)$ satisfies
\[
\pi^i_{\#}\gamma=\rho_i \mbox{ for } i=1,2.
\]
Owing to the regularisation by $V_\varepsilon$, one can show (generalising and adapting the computations in~\cite[Propositions 3.4 and 3.6]{blob_weighted_craig}) that the functional $\mf^\varepsilon$ satisfies a geodesic `above the tangent line' inequality~\cite[Proposition 2.8]{C17} in the sense that
\begin{equation}
    \label{eq:above_tan}
    \mf^\varepsilon(\rho_2) - \mf^\varepsilon(\rho_1) - \left.\frac{d}{d\alpha}\right|_{\alpha = 0}\mf^\varepsilon(\rho_\alpha) \ge - \frac{\lambda_F^\varepsilon}{2}d_W^2(\rho_1,\rho_2),
\end{equation}
where
\[
\lambda_F^\varepsilon: = -\frac{c_2\|D^2V_\varepsilon\|_{L^\infty}\|V_\varepsilon\|_{L^\infty}^{m-2}}{m-1}.
\]
The technical assumption $F'(0) = 0$ enters here in the computation of~\eqref{eq:above_tan}. As a consequence we infer $\lambda_F^\varepsilon$-convexity similar to~\cite[Proposition 3.11]{Patacchini_blob19} as well as a characterisation of the subdifferential, \cite[Proposition 3.12]{Patacchini_blob19} --- adapted to our functional, meaning that 
\[
\frac{\delta\mf^\varepsilon}{\delta\rho}=V_\varepsilon*F'(V_\varepsilon*\rho).
\]
As for the modulus of convexity $\lambda_F^\varepsilon$, similarly to \cite{blob_weighted_craig}, using~\ref{ass:F-m} and fixed $m>1$, we have
\[
\lambda_F^\varepsilon\approx-\varepsilon^{-2-d(m-1)},
\]
meaning that $\lambda_F^\varepsilon\to-\infty$ as $\varepsilon\to0$.
The information above are enough to prove existence of a unique gradient flow of $\mf^\varepsilon$, for $\varepsilon>0$ fixed, following \cite{AGS} and \cite[Section 5]{Patacchini_blob19} for regularised energies, thus proving~\Cref{thm:uniqueness}. We omit the details and refer the interested reader again to similar computations in~\cite{blob_weighted_craig}.

The regularisation of the energy by mollifiers $V_1$ satisfying~\ref{ass:v1} used in this manuscript, as well as in \cite{BE22}, allows to extend stability of gradient flows to the case $m>1$ --- further assuming $V_1$ is compactly supported for $1<m<2$ (c.f.~\Cref{sec:mle2}). The advantage of using convex energies is given by the possibility of using stability estimates with the $2$-Wasserstein distance so that one obtains a particle approximation when the number of particles involved depends on $\varepsilon$, i.e. $N=N(\varepsilon)$. This is a qualitative result, recently proved rigorously in~\cite[Theorem 1.4]{blob_weighted_craig}, for $m=2$. In our setting, we consider~\eqref{eq:nlie-class} as a continuity equation with velocity given by $-\nabla V_\varepsilon*F'(V_\varepsilon*\rho)$.  Then, under mild assumptions on the mollifer $V_1$ and the function $F$, the empirical measure $\rho^N_\varepsilon(t) = \frac{1}{N}\sum_{j=1}^N \delta_{x^j_\varepsilon(t)}$ is a weak solution to~\eqref{eq:nlie-class} provided the particles satisfy the following ODE system
\[
\dot{x}^i_\varepsilon(t) = - \nabla \int_{\R^d} V_\varepsilon(x^i_\varepsilon(t) - y) F'\left(
\frac{1}{N}\sum_{j=1}^NV_\varepsilon(y - x^j_\varepsilon(t))
\right)dy \quad \forall i=1,\dots, N.
\]
The regularisation in~\eqref{eq:nlie-class} is done in the same spirit as~\cite{blob_weighted_craig}, therefore an analogous version of~\cite[Theorem 1.4]{blob_weighted_craig} also holds true in our setting, with $\lambda_F^\varepsilon$ specified above and $m>1$. More precisely, as a consequence of the usual stability estimate for $\lambda$-gradient flows, \cite[Theorem 11.2.1]{AGS}, we know
\[
d_W(\rhoe(t),\rho_\varepsilon^N(t))\le e^{-\lambda_F^\varepsilon t} d_W(\rho_\varepsilon^N(0),\rho(0)).
\]
Therefore, if we assume that for $\varepsilon\to0$ there exists $N=N(\varepsilon)\to+\infty$ such that
\[
e^{-\lambda_F^\varepsilon t} d_W(\rho_\varepsilon^N(0),\rho(0))\to0,
\]
we infer the mean field limit since $\rhoe$ is converging to a weak solution of \eqref{eq:nonlinear-diffusion}. 

\begin{appendices}
\section{Technical proofs}
\label{sec:proofs}
We begin with an extension of the Dominated Convergence Theorem~\cite[Chapter 4, Theorem 17]{R88}. This is quoted with less generality for our purposes.
\begin{thm}[Extended Dominated Convergence Theorem]
\label{thm:EDCT}
Let $(f^\sigma)_{\sigma>0}$ and $(g^\sigma)_{\sigma>0}$ be sequences of measurable functions on a (Lebesgue) measurable set $X\subset\R^n$ such that $g^\sigma\ge 0$ and suppose that there exist measurable functions $f, \, g$ satisfying the following assumptions.
\begin{enumerate}
	\item $|f^\sigma(y)|\le g^\sigma(y)$ for all $\sigma>0$ and pointwise almost every $y\in X$.
	\item $f^\sigma(y) \to f(y)$ and $g^\sigma(y) \to g(y)$ pointwise almost every $y\in X$ as $\sigma\to 0$.
	\item $\int_{X}g^\sigma(y) dy \to \int_{X} g(y) dy$ as $\sigma\to 0$.
\end{enumerate}
Then, we have $\int_{X}f^\sigma(y) dy \to \int_{X} f(y)dy$ as $\sigma\to 0$.
\end{thm}
Next, we turn to estimates related to the JKO scheme with the regularised energy $\mf^\varepsilon$ and regularity of $V$ and $F$ under various assumptions.
\begin{rem}\label{rem:mom-ineq}
	From the definition of the $2$-Wasserstein distance and the inequality $|y|^2\le2|x|^2+2|x-y|^2$ it follows
	$$
	m_2(\rho_1)\le2m_2(\rho_0)+2d_W^2(\rho_0,\rho_1), \qquad \forall \rho_0,\rho_1\in\mptrd.
	$$
\end{rem}
For a function $F:[0,+\infty)\to (-\infty,+\infty]$ satisfying~\ref{ass:AGS-F}, its negative part can be estimated (c.f.~\cite[Remark 9.3.7]{AGS}) by
\[
F^-(s) \le c_1 s + c_2 s^{\alpha}, \quad \forall s\in[0,+\infty),
\]
for some constants $c_1, \, c_2 >0$ and we can take, without loss of generality, $\alpha \in \left(\frac{d}{d+2},1\right)$. With this notation, we have the following result.
\begin{lem}
\label{lem:lower-bound-F}
Suppose $F:[0,+\infty) \to (-\infty,+\infty]$ satisfies~\ref{ass:AGS-F} for some $\alpha \in \left(\frac{d}{d+2},1\right)$. Then, whenever $\rho\in \mptrd$ has a density with respect to Lebesgue, $F^-(\rho)\in L^1(\R^d)$. In particular, the functional $\mf$ admits the lower bound
\[
\mf[\rho] \ge -c_1 - c_2 \, C_{d,\alpha} \, (1 + m_2(\rho))^\alpha > -\infty,
\]
where $C_{d,\alpha}>0$ is an explicitly computable constant depending only on $d$ and $\alpha$. Moreover, for any $\varepsilon>0$ and $V_1$ with finite second moment as in~\ref{ass:v1}, the functional $\mf^\varepsilon$ admits the lower bound
\begin{equation}
\label{eq:low-bdd-fveps}
\mf^\varepsilon[\rho]\ge -c_1 - c_2 \tilde{C}_{d,\alpha}(1 + \varepsilon^2 m_2(V_1) + m_2(\rho))^\alpha,
\end{equation}
where $\tilde{C}_{d,\alpha} = 4C_{d,\alpha}$ but we suppress multiplicative constants and abuse notation by reusing $C_{d,\alpha}$.
\end{lem}
\begin{proof}
The proof follows~\cite[Remark 9.3.7]{AGS}. We have $\int_{\Rd}F^-(\rho(x))dx \le c_1 + c_2 \int_{\Rd}\rho^{\alpha}(x)dx$. Using H\"older's inequality and $\frac{2\alpha}{1-\alpha}>d$, we have
\begin{align*}
&\quad \int_{\Rd}\rho^{\alpha}(x)dx = \int_{\Rd}\rho^{\alpha}(x)(1+|x|)^{2\alpha}(1+|x|)^{-2\alpha}dx  	\\
&\le \left(
\int_{\Rd}(1+|x|)^2\rho(x)dx
\right)^{\alpha}\left( \int_{\Rd}(1+|x|)^\frac{-2\alpha}{1-\alpha} dx \right)^{1-\alpha} 	\\
&\le C_{d,\alpha}(2 + 2m_2(\rho))^\alpha < +\infty.
\end{align*}
The lower bound for $\mf$ readily follows. It remains to establish the lower bound for $\mf^\varepsilon$. Recall that since $\mf^\varepsilon[\rho] = \mf[V_\varepsilon * \rho]$, we have
\[
\mf^\varepsilon[\rho] \ge -c_1 - c_2 \, C_{d,\alpha} \, (1 + m_2(V_\varepsilon * \rho))^\alpha.
\]
It suffices to estimate $m_2(V_\varepsilon * \rho)$. Using Fubini, we have
\begin{align*}
		&\quad m_2(V_\varepsilon * \rho) = \int_{\Rd}|x|^2 \int_{\Rd}V_\varepsilon(x-y) \, d\rho(y) \, dx 	\\
		&\le 2\int_{\R^d} \left(\int_{\Rd}
		|x-y|^2 V_\varepsilon(x-y) dx
		\right) d\rho(y) + 2\int_{\Rd} \left(\int_{\Rd}
		V_\varepsilon(x-y) dx
		\right) |y|^2 d\rho(y) 	\\
		&= 2\int_{\Rd}\left(\int_{\Rd}|z|^2 \varepsilon^{-d}V_1(z/\varepsilon)dz
		\right)d\rho(y) + 2m_2(\rho) 	= 2\varepsilon^2 m_2(V_1) + 2m_2(\rho).
\end{align*}
\end{proof}
The lower bound for $\mf^\varepsilon$ is used to prove that one step of the JKO scheme~\eqref{eq:jko} is well-defined.
\begin{lem}
	\label{lem:one-step-JKO-F}
	Fix $\varepsilon>0$ and suppose $F: [0,+\infty) \to (-\infty,+\infty]$ satisfies~\ref{ass:AGS-F}. For fixed $\bar{\rho}\in \mptrd$, the functional $\rho \in \mptrd \mapsto \frac{d_W^2(\rho,\bar{\rho})}{2\tau} + \mf^\varepsilon[\rho] \in (-\infty,+\infty]$ admits minimisers whenever $\tau>0$ is chosen independently of $\varepsilon>0$ to satisfy
	\[
	\tau \le \frac{1}{2c_2 \, C_{d,\alpha}}.
	\]
\end{lem}
\begin{proof}
This proof follows the direct method.

\underline{Step 1 - Uniform lower bound:} We apply~\eqref{eq:low-bdd-fveps} from~\Cref{lem:lower-bound-F} to begin with
\begin{align*}
&\frac{d_W^2(\rho,\bar{\rho})}{2\tau} + \mf^\varepsilon[\rho]  \ge \frac{d_W^2(\rho,\bar{\rho})}{2\tau} - c_1 - c_2C_{d,\alpha}(1 + \varepsilon^2 m_2(V_1) + m_2(\rho))^\alpha.
\end{align*}
Using~\Cref{rem:mom-ineq}, we can replace $m_2(\rho)$ to further estimate
\begin{align*}
&\frac{d_W^2(\rho,\bar{\rho})}{2\tau} + \mf^\varepsilon[\rho]  \ge \frac{d_W^2(\rho,\bar{\rho})}{2\tau} - c_1 - c_2C_{d,\alpha} (1 + \varepsilon^2 m_2(V_1) + m_2(\bar{\rho}) +  d_W^2(\rho,\bar{\rho})),
\end{align*}
where we have also used $(1+|x|)^\alpha \le 1+|x|$ for $\alpha\in(0,1)$. Finally, the only dependence on $\rho$ is through $d_W^2(\rho,\bar{\rho})$. The coefficient of this term is positive when
\[
\frac{1}{2\tau} - c_2C_{d,\alpha}\ge 0 \iff 0 < \tau \le \frac{1}{2c_2 C_{d,\alpha}}.
\]
With this restriction on $\tau$, the lower bound reads
\[
\frac{d_W^2(\rho,\bar{\rho})}{2\tau} + \mf^\varepsilon[\rho]  \ge - c_1 - c_2C_{d,\alpha}(1 + \varepsilon^2m_2(V_1) + m_2(\bar{\rho})).
\]
\underline{Step 2 - The Direct Method:} From the previous step, we can find an infimising sequence $\rho_n\in\mptrd$ such that
\[
\frac{d_W^2(\rho_n,\bar{\rho})}{2\tau} + \mf^\varepsilon[\rho_n] \to \inf_{\mu\in\mptrd}\left\{ \frac{d_W^2(\mu,\bar{\rho})}{2\tau} + \mf^\varepsilon[\mu]\right\} >-\infty.
\]
As a convergent sequence, we also have
\[
\sup_{n\in\mathbb{N}}\left\{ \frac{d_W^2(\rho_n,\bar{\rho})}{2\tau} + \mf^\varepsilon[\rho_n]\right\} \le C < +\infty,
\]
for some constant $C>0$. By~\Cref{rem:mom-ineq}, the second moments are uniformly bounded by
\[
\sup_{n\in\mathbb{N}}m_2(\rho_n) \le 2m_2(\bar{\rho}) + C\tau.
\]
Therefore, by Prokhorov's theorem, there exists a subsequence, still labelled $\rho_n$, such that $\rho_n \rightharpoonup \rho$ (narrow convergence) for $\rho\in\mptrd$. Moreover, it is known (c.f.~\Cref{rem:F-assumptions}) that~\ref{ass:AGS-F} provides the required lower semicontinuity with respect to narrow convergence so that
\[
\frac{d_W^2(\rho,\bar{\rho})}{2\tau} + \mf^\varepsilon[\rho] \le \liminf_{n\to\infty}\left\{
\frac{d_W^2(\rho_n,\bar{\rho})}{2\tau} + \mf^\varepsilon[\rho_n]
\right\} = \inf_{\mu\in\mptrd}\left\{ \frac{d_W^2(\mu,\bar{\rho})}{2\tau} + \mf^\varepsilon[\mu]\right\}.
\]
\end{proof}
These results are enough to prove~\Cref{prop:en-ineq-mom-bound}.

\begin{lem}
\label{lem:bdd-comp-F-V}
Suppose $F:[0,+\infty) \to \R$ satisfies~\ref{ass:diff-F} and for $\varepsilon>0$, define $V_\varepsilon(x) = \varepsilon^{-d}V_1(x/\varepsilon)$ where $V_1$ satisfies~\ref{ass:v1}. Then, for any non-negative measure $\rho \in \mathcal{M}(\Rd)$ such that $\rho(\Rd) = C<+\infty$, we have the estimate 
\[
\|F'(V_\varepsilon*\rho(\cdot))\|_{L^\infty(\Rd)} \le \|F'(\cdot)\|_{L^\infty([0, \,C\|V_\varepsilon\|_{L^\infty}])}<+\infty.
\]
\end{lem}
\begin{proof}
We estimate the convolution of $V_\varepsilon$ against $\rho$ by
\begin{align*}
&\quad V_\varepsilon * \rho(x) = 
\int_{\Rd}V_\varepsilon(x-y)d\rho(y) \le C\|V_\varepsilon\|_{L^\infty}.
\end{align*}
Being $V_1\ge 0$ and $\rho$ a positive measure, the image of $V_\varepsilon* \rho(\cdot)$ is therefore contained in $[0,\, C\|V_\varepsilon\|_{L^\infty}]$. Being $F'$ continuous on $[0,+\infty)$, it is bounded on $[0,\, C\|V_\varepsilon\|_{L^\infty}]$, and the result is proved.
\end{proof}
\begin{lem}
	\label{lem:Lip-cty-conv}
For any $\varepsilon>0$, define $V_\varepsilon(x) = \varepsilon^{-d}V_1(x/\varepsilon)$ where $V_1$ satisfies~\ref{ass:v1}. Then, for any compact set $K\subset\Rd$ and $\rho\in\mprd$, there is a constant $C>0$ depending only on $\varepsilon, \, K, \, V_1$, and $m_1(\rho)$ such that
\[
|V_\varepsilon*\rho(x^1) - V_\varepsilon*\rho(x^2)| \le C |x^1 - x^2|, \quad \forall x^1,\, x^2\in K.
\]
Moreover, suppose $\rho^n\in\mprd$ is a sequence such that $\sup_{n\in\mathbb{N}}m_1(\rho^n) <+\infty$. Then, there is a subsequence of $\rho^n$ (still labelled $\rho^n$) such that $\rho^n \rightharpoonup \rho\in \mprd$ and the following limit holds
\[
\lim_{n\to\infty}\sup_{x\in K} |V_\varepsilon*\rho^n(x) - V_\varepsilon*\rho(x)| = 0.
\]
\end{lem}
\begin{proof}
We use the mean-value form of Taylor's theorem to estimate
\begin{align*}
&\quad \left|
\int_{\Rd}(V_\varepsilon(x^1-y) - V_\varepsilon(x^2-y)) d\rho(y)
\right| \le |x^1-x^2|\int_{\Rd}\int_0^1\left|
\nabla V_\varepsilon(tx^1 + (1-t)x^2-y)
\right|dt \, d\rho(y) 	\\
&\le C_{\varepsilon, \,V_1}|x^1-x^2|\int_{\Rd}\left(1 + |x^1| + |x^2| + |y|\right) \, d\rho(y) \le C_{\varepsilon, \,V_1}C_K|x^1-x^2|(1 + m_1(\rho)).
\end{align*}
In the last line, the constant $C_{\varepsilon, \, V_1}$ comes from the linear growth assumption of $\nabla V_1$ scaled with $\varepsilon$, and the constant $C_K$ absorbs bounds for $|x^1|$ and $|x^2|$ since $x^1, \, x^2\in K$.

Turning to the sequence $\rho^n$, standard tightness arguments using the uniform first moment bound yield $\rho\in\mprd$ such that $\rho^n\rightharpoonup \rho$ along a subsequence. As for the uniform convergence over $K$ along a subsequence, properties of the mollification by $V_\varepsilon$ yield that
\[
x\mapsto |V_\varepsilon*\rho^n(x) - V_\varepsilon*\rho(x)|
\]
is a continuous function for every $n\in\mathbb{N}$. Therefore, for every $n\in\mathbb{N}$, there is $x_K^n\in K$ such that
\[
\sup_{x\in K} |V_\varepsilon*\rho^n(x) - V_\varepsilon*\rho(x)| = |V_\varepsilon*\rho^n(x_K^n) - V_\varepsilon*\rho(x_K^n)|.
\]
Since $K$ is compact, there is a subsequence, still labelled $x_K^n$, such that $x_K^n \to x_K\in K$ as $n\to \infty$. Along this subsequence, we continue the previous estimate
\begin{align*}
&\quad \sup_{x\in K} |V_\varepsilon*\rho^n(x) - V_\varepsilon*\rho(x)| = |V_\varepsilon*\rho^n(x_K^n) - V_\varepsilon*\rho(x_K^n)| 	\\
&\le |V_\varepsilon*\rho^n(x_K^n) - V_\varepsilon*\rho^n(x_K)| + |V_\varepsilon*\rho^n(x_K) - V_\varepsilon*\rho(x_K)| + |V_\varepsilon*\rho(x_K) - V_\varepsilon*\rho(x_K^n)|.
\end{align*}
The first and third differences can be made arbitrarily small by the previous Lipschitz estimate owing to $\sup_{n\in\mathbb{N}}m_1(\rho^n)<+\infty$. For the second difference, remember that $V_\varepsilon(\cdot) \in C_b(\Rd)$ so certainly $V_\varepsilon(x_K - \cdot) \in C_b(\Rd)$ and $\rho^n\rightharpoonup\rho$ in duality against $C_b$.
\end{proof}

\begin{cor}
\label{cor:comp-f-conv}
In the setting of~\Cref{lem:Lip-cty-conv}, suppose $F$ satisfies~\ref{ass:diff-F}. Then, the composition $F'(V_\varepsilon*\rho(\cdot))$ is uniformly continuous on $K$. Moreover, 
\[
\lim_{n\to\infty}\sup_{x\in K} |F'(V_\varepsilon*\rho^n(x)) - F'(V_\varepsilon*\rho(x))| = 0.
\]
\end{cor}
\begin{proof}
Fix $\eta>0$, then since $F'\in C([0,\,+\infty))$, it is uniformly continuous on $[0, \, \|V_\varepsilon\|_{L^\infty}]$ so there exists $\delta_1>0$ such that
\begin{equation}
	\label{eq:F'-unif-cts}
\forall V^1, \, V^2 \in [0, \, \|V_\varepsilon\|_{L^\infty}], \quad |V^1 - V^2| < \delta_1 \implies |F'(V^1) - F'(V^2)| < \eta.
\end{equation}
Starting with the uniform continuity of $F'(V_\varepsilon* \rho(\cdot))$, for the $\delta_1$ in~\eqref{eq:F'-unif-cts}, define $\delta := \delta_1/C$ where $C>0$ is the Lipschitz constant in~\Cref{lem:Lip-cty-conv}. Then, whenever $x^1, \, x^2 \in K$ satisfy $|x^1 - x^2| < \delta$, the Lipschitz estimate from~\Cref{lem:Lip-cty-conv} implies
\[
|V_\varepsilon*\rho(x^1) - V_\varepsilon * \rho(x^2)| < \delta_1.
\]
As well, we can take $V^i = V_\varepsilon*\rho(x^i)$ for $i=1, \, 2$ in~\eqref{eq:F'-unif-cts} to conclude the uniform continuity. Note that $V^i \in [0, \, \|V_\varepsilon\|_{L^\infty}]$.

Turning to the limit, the uniform convergence from~\Cref{lem:Lip-cty-conv} gives the existence of $N\in\mathbb{N}$ such that for every $n\ge N$ and $x\in K$, we have
\[
|V_\varepsilon*\rho^n(x) - V_\varepsilon*\rho(x)| < \delta_1.
\]
We would like to apply~\eqref{eq:F'-unif-cts} and set $V^1 = V_\varepsilon*\rho^n(x)$ and $V^2 = V_\varepsilon*\rho(x)$. We need to verify that $V_\varepsilon*\rho^n(x), \, V_\varepsilon*\rho(x)\in [0, \, \|V_\varepsilon\|_{L^\infty}]$ for every $x \in K$ which is easily proven.
\end{proof}
\end{appendices}

\subsection*{Acknowledgements}
The authors would like to thank Martin Burger, David Gómez-Castro, and Giuseppe Savaré for helpful discussions on the topics of the manuscript. The authors were supported by the Advanced Grant Nonlocal-CPD (Nonlocal PDEs for Complex Particle Dynamics: Phase Transitions, Patterns and Synchronization) of the European Research Council Executive Agency (ERC) under the European Union’s Horizon 2020 research and innovation programme (grant agreement No. 883363). JAC was also partially supported by the EPSRC grant numbers EP/T022132/1 and EP/V051121/1. JW was supported by the Mathematical Institute Award at the University of Oxford.

\subsection*{Availability of data and materials}

Data sharing not applicable to this article as no datasets were generated or analysed during the current study.

\bibliography{main}
\bibliographystyle{abbrv}

\end{document}